\documentclass[twoside,a4paper,reqno,11pt]{amsart} 
\usepackage[top=28mm,right=30mm,bottom=28mm,left=30mm]{geometry}

\usepackage{amsfonts, amsmath, amssymb, mathrsfs, bm, latexsym, stmaryrd, array, hyperref, mathtools, bold-extra, xcolor, enumerate,booktabs}

\newcommand{\wy}{\widehat{y}}
\newcommand{\QQ}{\widetilde{Q}}
\newcommand{\RC}{\mathrm{RC}}
\renewcommand{\a}{\alpha}
\renewcommand{\b}{\beta}
\newcommand{\g}{\gamma}
\newcommand{\e}{\varepsilon}

\newcommand\bbF{\mathbb{F}}
\newcommand{\normeq}{\trianglelefteqslant}

\newcommand{\A}{\mathrm{A}}

\newcommand{\la}{\langle}
\newcommand{\ra}{\rangle}

\newcommand{\leqs}{\leqslant}
\newcommand{\geqs}{\geqslant}
\def\ge{\geqs}
\def\le{\leqs}

\newcommand{\Aut}{\mathrm{Aut}}
\newcommand{\Inn}{\mathrm{Inn}}
\newcommand{\Out}{\mathrm{Out}}

\newcommand{\PSL}{\mathrm{PSL}}
\newcommand{\GL}{\mathrm{GL}}
\newcommand{\PGL}{\mathrm{PGL}}
\newcommand{\SL}{\mathrm{SL}}

\newcommand{\PSp}{\mathrm{PSp}}
\newcommand{\Sp}{\mathrm{Sp}}

\newcommand{\POmega}{\mathrm{P\Omega}}
\newcommand{\LL}{\mathrm{L}}
\newcommand{\UU}{\mathrm{U}}

\newcommand{\PGU}{\mathrm{PGU}}

\newcommand{\GU}{\mathrm{GU}}

\newcommand{\Hol}{\mathrm{Hol}}

\def\Sym{\mathrm{Sym}}
\newcommand{\Sn}{\mathrm{S}}
\newcommand{\An}{\mathrm{A}}

\newcommand{\minPart}{\Gamma_{k, |T|}}

\makeatletter
\newcommand{\imod}[1]{\allowbreak\mkern4mu({\operator@font mod}\,\,#1)}
\makeatother

\newtheorem{theorem}{Theorem} 
\newtheorem*{conj*}{Conjecture}

\newtheorem{corol}{Corollary}

\newtheorem{thm}{Theorem}[section] 
\newtheorem{prop}[thm]{Proposition} 
\newtheorem{lem}[thm]{Lemma}
\newtheorem{cor}[thm]{Corollary}

\newtheorem*{prob*}{Problem}

\theoremstyle{definition}

\begin{document}

\title{Greedy bases and relational complexity of diagonal type groups}
\begin{abstract}
  A \emph{base} for a subgroup $G$ of $\Sym(\Omega)$ is a sequence of elements of $\Omega$ with trivial pointwise stabiliser. The size of the smallest base for $G$ is denoted $b(G)$. There is a natural greedy algorithm to compute a base for $G$, and it was conjectured by Cameron in 1999 that there exists an absolute constant $c$ such that if $G$ is primitive then any base returned by this algorithm has size at most $cb(G)$. In this paper we determine the size of 
  every
  base returned by the greedy algorithm when $G$ is a primitive group of diagonal type, and hence prove Cameron's conjecture for these groups. 

  The relational complexity $\RC(G)$ of $G$ is a measure of the way in which the orbits of $G$ on $\Omega^k$ for various $k$ determine the action of $G$ on $\Omega$. Very few precise values of relational complexity are known, and in particular it is not known which primitive groups have relational complexity $3$. In this paper we prove that if $G$ is primitive of diagonal type then $\RC(G) \geqs 4$, that this lower bound is attained by infinitely many such $G$, and that the relational complexity of the groups of diagonal type is unbounded. 
 \end{abstract}

\date{\today}

\author{Hong Yi Huang}
    \address{H.Y. Huang, Department of Mathematics, Southern University of Science and Technology, Shenzhen 518055, Guangdong, P. R. China}
    \curraddr{Alfréd Rényi Institute of Mathematics, Reáltanoda utca 13-15., H-1053, Budapest, Hungary}
    \email{11612012@mail.sustech.edu.cn}

\author{Colva M. Roney-Dougal}
    \address{C.M. Roney-Dougal, School of Mathematics and Statistics, University of St Andrews, St Andrews, UK}
    \email{Colva.Roney-Dougal@st-andrews.ac.uk}

\maketitle

\section{Introduction}

A \textit{base} for a subgroup $G$ of $\Sym(\Omega)$ is a sequence of elements of $\Omega$ with trivial pointwise stabiliser in $G$. For finite $G$ the \textit{base size} $b(G)$ is the minimal size of a base for $G$. This classical invariant of permutation groups has been studied since the 19th century, and gained particular importance in computational group theory \cite{S_compu}, as the running time of many permutation group algorithms is a function of the size of a known base. A base $(\b_1, \ldots, \b_k)$ is \emph{irredundant} if the pointwise stabiliser in $G$ of $\b_1, \ldots, \b_{i+1}$ is a proper subgroup of the pointwise stabiliser of $\b_1, \ldots, \b_i$, for all $i \in \{1, \ldots, k-1\}$. The size of the largest irredundant base is denoted by $I(G)$: it is straightforward to see that $I(G) \leqs b(G) \log n$ (our logarithms are to the base 2 unless otherwise specified). 

In general, $b(G)$ can be as large as $|\Omega| -1$, but much smaller bounds are known for most types of primitive groups.  The problem of determining $b(G)$ has been considered for all five O'Nan--Scott classes. For example, the precise base sizes of primitive groups of diagonal type are determined by the first author \cite{H_diag} (see Theorem~\ref{t:H_diag} below), and so far this is the only O'Nan--Scott family for which the base sizes are known in all cases. 

There are no efficient algorithms
for computing $b(G)$, or  for constructing a base of minimal size. Indeed, Blaha \cite{Blaha} proves that computing a base of size $b(G)$ is NP-hard. To find a base of ``relatively small" size, 
there is a natural greedy algorithm, first systematically studied by Blaha. Let $\beta_1$ be any point in any longest orbit of $G$. Now define a sequence $(\b_i)_{i\geqslant 1}$ of points of  $\Omega$, terminating as soon as a base is reached, where $\b_m$ lies in a longest orbit of the pointwise stabiliser $G_{\b_1, \ldots, \b_{m-1}}$.
The sequence $(\b_i)_{i\geqslant 1}$ is called a \textit{greedy base} for $G$, and the \textit{greedy base size} $\mathcal{G}(G)$ is defined to be the largest size of a greedy base for $G$. Blaha \cite{Blaha} shows that $\mathcal{G}(G)$ is ``small" in the sense that there is an absolute constant $d$ such that
\begin{equation}\label{e:loglog}
    \mathcal{G}(G)\leqslant db(G)\log\log|\Omega|
\end{equation}
for every permutation group $G\leqslant\mathrm{Sym}(\Omega)$. And in the same paper, he proves that for a fixed integer $k$ and a sufficiently large integer $n$, there exists a permutation group $G$ of degree $n$ such that $b(G) = k$ and $\mathcal{G}(G) \geqslant \frac{1}{5}k\log\log n$, so the bound \eqref{e:loglog} is best possible. 

For primitive groups, Cameron conjectured in 1999 \cite{CameronBook} that a substantially tighter bound than \eqref{e:loglog} holds.
\begin{conj*}[Cameron's Greedy Conjecture]
    There is an absolute constant $c$ such that $\mathcal{G}(G)\leqs cb(G)$ for every finite primitive group $G$.
\end{conj*}

Since every greedy base is irredundant, $\mathcal{G}(G) \leqs I(G)$. Building on \cite{GillLodaSpiga}, Kelsey and the second author showed in \cite{KelseyCMRD} that if $G$ is primitive and not large base then $I(G) \leqs 5 \log n$, which automatically bounds $\mathcal{G}(G)$ for these groups,  but does not directly address Cameron's Greedy Conjecture. The conjecture has been proved for many almost simple primitive groups \cite{d_greedy,dR_greedy}, and been proved (with  $c = 1$) for groups of odd order \cite{BdR_odd}.

In this paper, we focus on primitive groups of diagonal type. In this setting, $G\leqs \mathrm{Sym}(\Omega)$ has socle $T^k$, where $T$ is a non-abelian simple group and $k\geqs 2$ is an integer. Here $|\Omega| = |T|^{k-1}$ and 
\begin{equation*}
    T^k\normeq G\leqs T^k.(\Out(T)\times \Sn_k).
\end{equation*}
The subgroup $P\leqs \Sn_k$ induced by the action of $G$ by conjugation on the set of factors of $T^k$ is called the \textit{top group} of $G$. We  define
\begin{equation*}
    Q = \{(1,\dots,1)\sigma : \sigma\in \Sn_k\} \cap G = \Sn_k \cap G.
\end{equation*}
Note that $Q$ is naturally isomorphic to a subgroup of the top group $P$.

For groups of diagonal type, Gill, Lod\'a and Spiga prove in \cite{GillLodaSpiga} that $I(G) \leqs \log |\Omega| + 1$, giving the same upper bound on $\mathcal{G}(G)$.  Our first main theorem is the following, which gives the first O'Nan-Scott family for which $\mathcal{G}(G)$ is known in all cases. 
\begin{theorem}\label{thm:greedy}
    Let $G$ be a primitive group of diagonal type with socle $T^k$ and top group $P$, and let $\ell  = \lceil \log_{|T|} k \rceil$. 
    \begin{itemize}\addtolength{\itemsep}{0.2\baselineskip}
		\item[{\rm (i)}] If $k = 2$, then $\mathcal{G}(G)\in\{3,4\}$, with $\mathcal{G}(G) = 4$ if and only if $T\in\{ \An_5, \An_6\}$ and $G = T^2.(\Out(T)\times \Sn_2)$.
		\item[{\rm (ii)}] If $P\notin \{\An_k, \Sn_k\}$, then $\mathcal{G}(G) = 2$.
		\item[{\rm (iii)}] For $k\geqs 3$, if $P\in\{ \An_k, \Sn_k\}$ then $\mathcal{G}(G)\in\{\ell+1,\ell+2\}$. Moreover, $\mathcal{G}(G) = \ell+2$ if and only if either $k = |T|^\ell$, or $k\in\{|T|^\ell-2,|T|^\ell-1\}$ and $Q = \An_k$.
	\end{itemize}
    In addition, every greedy base for $G$ has size $\mathcal{G}(G)$.
\end{theorem}

Combining Theorem~\ref{thm:greedy} with the values of $b(G)$ found in \cite{H_diag} (see Theorem~\ref{t:H_diag}), we get the following corollary.

\begin{corol}\label{corol:GgBg}
    Let $G$ be a primitive group of diagonal type with socle $T^k$ and top group $P$, and let $\ell  = \lceil \log_{|T|} k \rceil$.  Then $\mathcal{G}(G) \in \{b(G),b(G)+1\}$, with $\mathcal{G}(G) = b(G)+1$ if and only if $P\in\{\An_k,\Sn_k\}$ and one of the following holds:
    \vspace{1mm}
		\begin{itemize}\addtolength{\itemsep}{0.2\baselineskip}
			\item[{\rm (a)}]  $k = |T|^2-2$, $Q = \Sn_k$, and either $T\notin\{\An_5,\An_6\}$ or $G\ne T^k.(\Out(T)\times \Sn_k)$: here $\mathcal{G}(G) = 4$;
            \item[{\rm (b)}] $k = |T|^\ell-2$ with $\ell\geqs 3$ and $Q = \Sn_k$: here $\mathcal{G}(G) = \ell+2$;
            \item[{\rm (c)}] $k = |T|^\ell$ with $\ell \geqs 2$ and $Q = \An_k$: here $\mathcal{G}(G) = \ell+2$.
		\end{itemize}
    In particular, for all $b\geqs 3$, there exist infinitely many primitive groups $G$ of diagonal type such that $b(G) = b < \mathcal{G}(G)$.
\end{corol}

If a transitive group $G$ has base size $2$, then in particular $G$ has a regular suborbit, and so $\mathcal{G}(G) = 2$ as well. Combining this observation with Corollary~\ref{corol:GgBg} proves
Cameron's Greedy Conjecture for primitive groups of diagonal type.

\begin{corol}
    Let $G$ be a primitive group of diagonal type. Then $\mathcal{G}(G)\leqs\frac{4}{3}b(G)$, with equality attained infinitely often.
\end{corol}
 
In a different direction, relational complexity is a topic connecting permutation groups with the study of homogeneous structures in model theory, going back to pioneering work of Lachlan and Cherlin \cite{LachlanICM, CherlinGelfand}. 
For much more background, including an explanation of the model-theoretic connection, see \cite{GLS_binary}. Let $s \leqs t$ be positive integers, and let $\Lambda = (\lambda_1, \ldots, \lambda_t), \Sigma = (\sigma_1, \ldots, \sigma_t) \in \Omega^t$. Then $\Lambda$ and $\Sigma$ are \emph{$s$-subtuple complete} with respect to a subgroup $G$ of $\Sym(\Omega)$, and we write $\Lambda \sim_s \Sigma$, if for every subset of $s$ indices $i_1, \ldots, i_s$ there exists an element $g \in G$ such that $\lambda_{i_j}^g = \sigma_{i_j}$ for all $j \in \{1, \ldots, s\}$. The \emph{relational complexity} of $G$, denoted $\RC(G)$, is the smallest $s \geqs 2$ such that for all $t \geqs s$ and all $\Lambda, \Sigma \in \Omega^t$, if $\Lambda \sim_s \Sigma$ then $\Lambda \in \Sigma^G$. 

Cherlin conjectured in \cite{CherlinGelfand} a classification of the finite primitive permutation groups $G$ for which $\RC(G) = 2$, and in a dramatic breakthrough in \cite{GLS_binary}, Gill, Liebeck and Spiga proved this conjecture: the only examples are the symmetric group in its natural action,  or of affine type. The alternating group $\An_n$ in its natural action has relational complexity $n - 1$, so almost simple groups can achieve all possible relational complexities. It is well known (see for example \cite[(1.1) and (1.2)]{GillLodaSpiga}) that $\RC(G) \leqs I(G)+1$, so the previously-mentioned bounds on $I(G)$ immediately yield bounds on $\mathcal{G}(G)$.  Recently, Gill asked what lower and upper bounds could be placed on the relational complexity of groups in each of the remaining O'Nan--Scott classes. In this paper we answer this question completely for primitive groups of diagonal type. 

Firstly, we show that there are no diagonal type groups with relational complexity $3$. 
\begin{theorem}\label{thm:RC4}
    Let $G$ be a primitive group of diagonal type. Then $\mathrm{RC}(G) \geqs 4$.
\end{theorem}
This bound is best possible, as there are infinitely many primitive groups $G$ of diagonal type with $\RC(G) = 4$ (see Proposition \ref{p:RC4_inf}).

Secondly, we show that the relational complexity of the primitive groups $G$ of diagonal type is unbounded. Recall that Gill, Lod\`a and Spiga proved in \cite{GillLodaSpiga} that each such $G$ of degree $n$ satisfies $\RC(G) \leqs \log n + 1$. We prove a lower bound that occurs infinitely often that is asymptotically close to this.

\begin{theorem}\label{thm:RClog}
There are infinitely many values of $n$ for which there exists at least one primitive group $G$ of diagonal type of degree $n$ such that $\RC(G) \geqs \frac{1}{2} \log n/\log \log n$.
\end{theorem}

\subsection*{Structure of the paper}

In Section~\ref{sec:prelim}, we give basic notation for diagonal type primitive groups. Section~\ref{sec:k=2} is devoted to the proof of Theorem~\ref{thm:greedy}(i). Lemma~\ref{l:k=2_probability} is a key ingredient, and is a new probabilistic method to find a greedy base of size $3$.  Theorem~\ref{thm:greedy}(ii) is an immediate consequence of Theorem~\ref{t:H_diag}(ii), as $b(G) = 2$ implies $\mathcal{G}(G) = 2$. In Section~\ref{sec:k_ge_3} we prove Theorem~\ref{thm:greedy}(iii) and complete the proof of Theorem~\ref{thm:greedy}. We determine the minimum two-point stabiliser by working with \cite[Theorem 4]{H_diag} (see Lemmas~\ref{l:two_stab_sigma} and \ref{l:greedy_two_stab}), and a detailed analysis of partition stabilisers in alternating and symmetric groups is used to determine greedy bases. Finally, in Section~\ref{sec:RC}, we establish Theorems~\ref{thm:RC4} and \ref{thm:RClog}.

\subsection*{Acknowledgments}

The first author thanks the London Mathematical Society for their support as an LMS Early Career Research Fellow at the University of St Andrews in 2025.

\section{Preliminaries}\label{sec:prelim}

In this section we give more information about the structure of diagonal type groups, and fix some related notation.
Throughout the paper, we let $k\geqs 2$ be an integer and let $T$ be a non-abelian finite simple group. All of the groups we define depend on $k$ and $T$, but we shall omit $k$ and $T$ from our notation when there is no room for confusion. 

Define 
\begin{equation*}
\begin{aligned}
&W_0 =\{(\varphi_1,\dots,\varphi_k)\sigma\in \Aut(T)\wr_k \Sn_k: \Inn(T)\varphi_1 = \Inn(T)\varphi_i \mbox{ for all }i\},\\
&D_0 =\{(\varphi,\dots,\varphi)\sigma\in \Aut(T)\wr_k \Sn_k\} \leqs W_0,\\
&\Omega =[W_0(k,T):D_0].
\end{aligned}
\end{equation*}
Then $|\Omega| = |T|^{k-1}$ and $W_0 = T^k.(\Out(T)\times  \Sn_k)$ acts faithfully on $\Omega$. A group $G\leqs\mathrm{Sym}(\Omega)$ is of \textit{diagonal type} if
$T^k\normeq G\leqs W_0$.
Recall from the introduction the definition of the \emph{top group} $P$ of $G$, and notice that
\begin{equation}\label{e:diag}
T^k\normeq G\leqs W:= T^k.(\Out(T)\times P) \leqs W_0.
\end{equation}
It is well known (see, for example \cite[Theorem 4.5A]{DixonMortimer}) that $G$ is primitive if and only if either $P$ is primitive on $[k] = \{1,\dots,k\}$ or $k = 2$, and that if $G$ is primitive then $T^k$ is the socle of $G$. Let
\begin{equation*}
D:=G \cap D_0 \leqs  \{(\varphi,\dots,\varphi)\sigma:\varphi\in\Aut(T),\sigma\in P\}.
\end{equation*}
We shall generally identify $\Inn(T)$ with $T$, so that elements of $\Inn(T) \wr_k \Sn_k$ are written $(t_1, \ldots, t_k) \sigma$. Hence in this context $t$ denotes  the conjugation map $\varphi_t: T \rightarrow T, \ x \mapsto t^{-1} x t$. This enables us to make the identification
\begin{equation*}
\Omega = \{D(t_1,\dots,t_k):t_1,\dots,t_k\in T\}.
\end{equation*}
The action of $G$ on $\Omega$ is given by
\begin{equation*}
D(t_1,\dots,t_k)^{(\varphi_1,\dots,\varphi_k)\sigma} = D((t_{1^{\sigma^{-1}}})^{\varphi_{1^{\sigma^{-1}}}},\dots,(t_{k^{\sigma^{-1}}})^{\varphi_{k^{\sigma^{-1}}}}).
\end{equation*}
Notice that $ \a^{-1}\varphi_t\a = \varphi_{t^\a}$ for all $\varphi_t\in \Inn(T)$ and $\a \in \Aut(T)$,
so that for any element $(\varphi,\dots,\varphi)\sigma\in D$, 
\begin{equation}\label{e:GD}
D(t_1,\dots,t_k)^{(\varphi,\dots,\varphi)\sigma} = D((t_{1^{\sigma^{-1}}})^{\varphi},\dots,(t_{k^{\sigma^{-1}}})^{\varphi}), \quad \mbox{ so } G_D = D. 
\end{equation}

The base sizes of diagonal type groups were determined in \cite[Theorem 3]{H_diag}, extending earlier work of Fawcett \cite{F_diag}.
\begin{thm}[{\cite[Theorem 3]{H_diag}}]
	\label{t:H_diag}
	Let $G$ be a primitive group of diagonal type with socle $T^k$ and top group $P$, and let $\ell  = \lceil \log_{|T|} k \rceil$.
	\begin{itemize}\addtolength{\itemsep}{0.2\baselineskip}
		\item[{\rm (i)}] If $k = 2$, then $b(G)\in\{3,4\}$, with $b(G) = 4$ if and only if $T\in\{ \An_5, \An_6\}$  and $G = T^2.(\Out(T)\times \Sn_2)$. 
		\item[{\rm (ii)}] If $P\notin \{\An_k, \Sn_k\}$, then $b(G) = 2$.
		\item[{\rm (iii)}] For $k\geqs 3$, if $P\in\{ \An_k, \Sn_k\}$ then $b(G)\in\{\ell+1,\ell+2\}$. Moreover, $b(G) = \ell+2$ if and only if one of the following holds:
		\vspace{1mm}
		\begin{itemize}\addtolength{\itemsep}{0.2\baselineskip}
			\item[{\rm (a)}] $k = |T|$ (so $\ell = 1$);
			\item[{\rm (b)}] $k\in\{|T|-2,|T|^\ell-1,|T|^\ell\}$ and $Q = \Sn_k$;
			\item[{\rm (c)}] $T\in\{ \An_5, \An_6\}$, $k = |T|^2-2$ (so $\ell = 2$) and $G = T^k.(\Out(T) \times  \Sn_k)$. 
		\end{itemize}
	\end{itemize}
\end{thm}

\section{Greedy bases for groups with $k = 2$}\label{sec:k=2}

In this section, we shall prove Theorem~\ref{thm:greedy}(i). We shall therefore assume throughout that $k = 2$. Here, by Theorem~\ref{t:H_diag}(i),  we know that $b(G)\in\{3,4\}$, with $b(G) = 4$ if and only if $G = T^2.(\Out(T)\times  \Sn_2)$ with $T\in\{ \An_5, \An_6\}$. We will show that the same is true for $\mathcal{G}(G)$. 

Since $k = 2$, our notation from Section~\ref{sec:prelim} can be simplified.
The \textit{holomorph} of $T$ is $\Hol(T) = T{:}\Aut(T)$. 
For the rest of this section we shall identify $\Omega$ with $T$ and $W_0 = W_0(2, T) = T^2.(\Out(T)\times  \Sn_2)$ with $\la\Hol(T),\sigma\ra$, where $\sigma$ is the inversion map on $T$. For all $t \in T = \Omega$, the action is therefore given by
\begin{equation}\label{e:k=2_act}
t^\sigma = t^{-1} \quad \mbox{ and } \quad t^{g\varphi} = (g^{-1}t)^\varphi
\end{equation}
for any $g\varphi \in T{:}\Aut(T) = \Hol(T)$. So \eqref{e:diag} simplifies to
\begin{equation*}
T{:}\Inn(T)\leqs G\leqs W \leqs W_0.
\end{equation*}
Notice that $P \in \{1, \Sn_2\}$, with $W = \Hol(T)$ if $P = 1$ and $W = W_0$ if $P = \Sn_2$.

For an element $t\in T$ and a subgroup $A$ of $\Aut(T)$,  the \emph{invertiliser} of $x$ in $A$ is
\begin{equation*}
I_{A}(t) = \{\varphi\in A:t^\varphi = t \mbox{ or }t^{-1}\}.
\end{equation*}
Note that $|C_{A}(t)|\leqs |I_{A}(t)|\leqs 2|C_{A}(t)|$, with the first inequality strict if and only if $t^{-1}\in t^{A}$ and $|t| > 2$.

\begin{lem}
    \label{l:k=2_stabiliser}
    Let $x,y\in T$. Then the following properties hold.
    \begin{itemize}\addtolength{\itemsep}{0.2\baselineskip}
		\item[{\rm (i)}] The two-point stabiliser $G_{1,x} = (C_{\Aut(T)}(x)\cup \{\varphi\sigma:\varphi\in \Aut(T), \ x^{\varphi} = x^{-1}\})\cap G$.
		\item[{\rm (ii)}] If $I_{\Aut(T)}(x)\cap I_{\Aut(T)}(y) = 1$ then $(1,x,y)$ is a base for $G$.
		\item[{\rm (iii)}] If $P = \Sn_2$ and $|G_{1,x}|$ is minimal amongst all two-point stabilisers of $G$, then $|I_{\Aut(T)}(x)| \leqs |I_{\Aut(T)}(y)|\cdot |\Out(T)|$.
	\end{itemize}
\end{lem}

\begin{proof}
    Part (i) is an easy exercise, and Part (ii) follows from Part (i) by considering $G_{1,x} \cap G_{1, y}$. Now suppose $P = \Sn_2$ and $|G_{1,x}|$ is minimal. Then $|W_0:G|\leqs |\Out(T)|$, so
    \begin{equation*}
    \begin{aligned}
        |I_{\Aut(T)}(x)| = |(W_0)_{1,x}| &\leqs |G_{1,x}|\cdot |\Out(T)|\\&\leqs |G_{1,y}|\cdot |\Out(T)| \\&\leqs |(W_0)_{1,y}|\cdot |\Out(T)| = |I_{\Aut(T)}(y)|\cdot |\Out(T)|,
    \end{aligned}
    \end{equation*}
as required.
\end{proof}

We first prove that Theorem~\ref{thm:greedy}(i) holds when
$P = 1$ or  $T$ is alternating or sporadic.

\begin{lem}
	\label{l:k=2_greedy_alt_spo}
	Suppose $k = 2$, and $P = 1$ or $T$ is alternating or sporadic. Then $\mathcal{G}(G) = b(G) \in \{3, 4\}$, with $\mathcal{G}(G) = 4$ if and only if $T \in \{\An_5, \An_6\}$ and $G = T^2.(\Out(T) \times \Sn_2)$.
\end{lem}

\begin{proof}
    If $T  \in\{ \An_5, \An_6\}$ then this can be checked using \textsf{GAP} \cite{GAP} or {\sc Magma} \cite{MAGMA}, so assume otherwise. Then $b(G) = 3$ by Theorem~\ref{t:H_diag}(i), and 
	by \cite[Propositions 3.8, 3.14 and 3.15]{FHLR_Saxl} the group $G$ has the property that every pair of distinct points in $\Omega$ can be extended to a base of size three. Hence each two-point stabiliser has a regular orbit, so $\mathcal{G}(G) = 3$.
\end{proof}

\subsection{The case $T = \LL_2(q)$}
Next, we prove Theorem~\ref{thm:greedy}(i) for groups $T \cong  \LL_2(q)$, as our arguments here are different from those for higher-rank groups.

We first record some information about conjugacy classes. The following result 
is classical, for a recent reference covering this and much more, see \cite[Chapter 2]{deFranceschiLiebeckOBrien}. 

\begin{lem}
    \label{l:k=2_inv_L2}
    Suppose $T = \LL_2(q)$, with $q = p^f$ for some prime $p$,  $q\geqs 7$ and $q\ne 9$. Let $L$ be almost simple with socle $T$, let $L_0 = L\cap \PGL_2(q)$, let $\ell = |L_0:T|$, and let $x\in T$ be a non-identity element of order coprime to $p$. Then
    \begin{itemize}\addtolength{\itemsep}{0.2\baselineskip}
        \item[{\rm (i)}] the order $|x|$ divides $q \pm 1$: we let $\e \in \{1, -1\}$ be such that $|x|$ divides $q - \e$ if $|x| > 2$, and $q \equiv \e \pmod 4$ if $|x| = 2$;
        \item[{\rm (ii)}] there exists $z\in\PGL_2(q)$ of order $q-\e$ such that $x\in\la z\ra$;
        \item[{\rm (iii)}] there exists $g\in T$ such that $z^g = z^{-1}$ (and so $x^g = x^{-1}$);
        \item[{\rm (iv)}] $I_{L_0}(x) = \la z^{(2,q-1)/\ell}\ra{:}\la g\ra\cong D_{\ell(q-\e)}$;
        \item[{\rm (v)}] if $|x| = (q-\e)/(2,q-1)$ then $C_L(x) = \la z^{(2,q-1)/\ell}\ra$,
        and $I_L(x) = I_{L_0}(x)$.
   \end{itemize}
        Furthermore, every element of $T$ of order divisible by $p$ has order exactly $p$.
\end{lem}

\begin{lem}\label{l:k=2_L2_pre}
    Suppose $T = \LL_2(q)$, $q\geqs 7$ and $q\ne 9$, with $q = p^f$ for some prime $p$. Let $\mathcal{O}$ be a subset of the divisors of $|T|$ containing $p$; all divisors of $(q-1)/(2,q-1)$ other than $2$; and, if $q\equiv 1\pmod 4$, the integer $2$. Let $x\in T$ have order in $\mathcal{O}$. Then there exists $y\in T$ of order $(q-1)/(2,q-1)$ such that $(1,x,y)$ is a base for any group of diagonal type with socle $T^2$.
\end{lem}

\begin{proof}
    If $|x| = p$ then we may assume $x$ is the image in $T$ of
    \begin{equation*}
        \widehat{x} = 
        \begin{pmatrix}
            1 & \mu \\
            0 & 1
        \end{pmatrix} \in\SL_2(q)
    \end{equation*}
    for some $\mu\in\mathbb{F}_q^\times$. Take $\widehat{y} = \mathrm{diag}(\lambda,\lambda^{-1})$ for some primitive element $\lambda$ of $\mathbb{F}_q$, and let $y$ be the image in $T$ of $\widehat{y}$. Then, as explained in the proof of \cite[Lemma 3.20]{FHLR_Saxl}, the intersection $C_{\Aut(T)}(x)\cap C_{\Aut(T)}(y) = 1$ and there is no $g\in \Aut(T)$ such that $(x,y)^g = (x^{-1},y^{-1})$. By Lemma~\ref{l:k=2_stabiliser}(i) the group 
    \[\begin{array}{rl} G_{1,x, y} & = (C_{\Aut(T)}(x) \cap C_{\Aut(T)}(y) \cap  G) \cap (\{\varphi \sigma: \varphi \in \Aut(T), x^\varphi = x^{-1}, y^{\varphi} = y^{-1}\} \cap G)\\& = 1 \end{array}.\] 

    Hence we may assume that $m \ne p$. Then either $m = 2$ and $q\equiv 1\pmod 4$ (so $\e = 1$ in the notation of Lemma~\ref{l:k=2_inv_L2}(i)), or $m\ne 2$ and $m$ divides $(q-1)/(2,q-1)$.  By Lemma~\ref{l:k=2_inv_L2}(iv), the group $K:=I_{\PGL_2(q)}(x) \cong D_{2(q-1)}$, which is a maximal subgroup of $\PGL_2(q)$ of type $\GL_1(q)\wr S_2$. Let $z\in K$ be an element of order $(q-1)/(2,q-1)$. Then Lemma~\ref{l:k=2_inv_L2}(iv) and (v) show  that $I_{\Aut(T)}(z) = I_{\PGL_2(q)}(z) = K$. By \cite[Theorem 2.1]{J_base}, there exists $g\in \PGL_2(q)$ such that $K\cap K^g = 1$, whence
    \begin{equation*}
    \begin{aligned}
        1 = K\cap K^g &= I_{\PGL_2(q)}(x)\cap I_{\Aut(T)}(z^g) \\
        &= (I_{\Aut(T)}(x)\cap \PGL_2(q))\cap I_{\Aut(T)}(z^g) \\
        &= I_{\Aut(T)}(x)\cap (\PGL_2(q)\cap I_{\Aut(T)}(z^g)) \\
        &= I_{\Aut(T)}(x)\cap I_{\PGL_2(q)}(z^g) 
        = I_{\Aut(T)}(x)\cap I_{\Aut(T)}(z^g).
    \end{aligned}
    \end{equation*}
    The result now follows from Lemma~\ref{l:k=2_stabiliser}(ii), with $y = z^g$.
\end{proof}

\begin{lem}
    \label{l:k=2_L2}
    Suppose $T = \LL_2(q)$, with $q\geqs 7$ and $q\ne 9$. Then every primitive diagonal type group $G$ with socle $T^2$ satisfies $\mathcal{G}(G) = 3$.
\end{lem}

\begin{proof}
By Theorem~\ref{t:H_diag}, the base size $b(G) = 3$, so it suffices to show that $\mathcal{G}(G) \leqs 3$. Recall the set $\mathcal{O}$ from Lemma~\ref{l:k=2_L2_pre}.
    Let $y\in T$ be an element of order $(q-1)/(2,q-1)$, and let $x\in T$ be an element of order $m \not\in \mathcal{O}$. We shall show that $|G_{1, x}| > |G_{1, y}|$: it then follows that the greedy algorithm will pick a second base point whose order lies in the set $\mathcal{O}$, and hence $\mathcal{G}(G) = 3$ by Lemma~\ref{l:k=2_L2_pre}. 
    By Lemma~\ref{l:k=2_inv_L2}, either $m = 2$ and $q\equiv 3\pmod 4$, or $m > 2$ and $m$ divides $q+1$. 
    
    Recall from \eqref{e:k=2_act} that $\sigma$ is the inversion map on $T$. We divide the proof into four cases:
    \begin{itemize}\addtolength{\itemsep}{0.2\baselineskip}
		\item[{\rm (a)}] There is no $\rho\in\Aut(T)$ such that $\rho\sigma\in G$.
        \item[{\rm (b)}] There is an element $\rho\in\Aut(T)$ such that $\rho\sigma\in G$, but there is no $\rho\in\PGL_2(q)$ such that $\rho\sigma\in G$.
		\item[{\rm (c)}] There is an element $\rho\in\PGL_2(q)$ such that $\rho\sigma\in G$, but $\sigma\notin G$.
        \item[{\rm (d)}] $\sigma\in G$.
    \end{itemize}
    
    In Case (a), the group $G\leqs \Hol(T)$, so $P = 1$ and Lemma~\ref{l:k=2_greedy_alt_spo} shows that $\mathcal{G}(G) = 3$.

    For Case (b), 
     Lemma~\ref{l:k=2_inv_L2}(v) (with $y$ in place of $x$), shows that every element in $\Aut(T)$ inverting $y$ lies in the coset $gC_{\PGL_2(q)}(y)\subseteq \PGL_2(q)$, where $g\in T$ is an element inverting $y$. Hence
    \begin{equation*}
        \{\varphi\sigma : \varphi  \in \Aut(T),  \ y^{\varphi} = y^{-1}\}\cap G = \{\varphi\sigma: \varphi \in \PGL_2(q),  \ y^{\varphi} = y^{-1}\}\cap G = \emptyset,
    \end{equation*}
    by  our assumptions on $G$. 
    Thus Lemma~\ref{l:k=2_stabiliser}(i) shows that 
    \begin{equation*}
        |G_{1,y}| = |C_{\Aut(T)}(y)\cap G| = |C_{\PGL_2(q)}(y)\cap G| = 
        \begin{cases}
            2(q-1) & \mbox{if }\PGL_2(q)\leqs G\\
            q-1 & \mbox{otherwise.}
        \end{cases}
    \end{equation*}
    However,
    \begin{equation*}
        |G_{1,x}| \geqs |C_{\PGL_2(q)}(x)\cap G| = 
        \begin{cases}
            2(q+1) & \mbox{if } \PGL_2(q)\leqs G\\
            q+1 & \mbox{otherwise}
        \end{cases}
        > |G_{1,y}|.
    \end{equation*}

    For Case (c), notice that $\PGL_2(q)\not\leqs G$. First assume that $m>2$, so that $m$ divides $q+1$. Then Lemma~\ref{l:k=2_inv_L2}(iv) with $L_0 = \PGL_2(q)$ shows that $C_{\PGL_2(q)}(x) \cong C_{q+1}$ and there exists $w\in T$ such that $x^w = x^{-1}$. Thus
    \begin{equation*}
        \{\varphi \in \PGL_2(q)\setminus T : x^\varphi = x^{-1}\} = wC_{\PGL_2(q)}(x)\setminus wC_{T}(x)
    \end{equation*}
    has size $|C_{T}(x)| = (q+1)/(2,q-1)$. It follows that
    \begin{equation*}
        |G_{1,x}| \geqs 2(q+1)/(2,q-1).
    \end{equation*}
    Arguing as in Case (b) shows that
    \begin{equation*}
        |\{\varphi \in\PGL_2(q)\setminus T :y^\varphi = y^{-1}\}| = |gC_{\PGL_2(q)}(y)\setminus gC_{T}(y)| = (q-1)/(2,q-1).
    \end{equation*}
    Now Lemma~\ref{l:k=2_stabiliser}(i) shows that
    \begin{equation*}
    \begin{aligned}
        |G_{1,y}| &= |(C_{\Aut(T)}(y)\cup \{\varphi\sigma : \varphi \in \Aut(T), \  y^{\varphi} = y^{-1}\})\cap G| \\
        &= |C_{\PGL_2(q)}(y)\cap G| + |\{\varphi\sigma : \varphi\in\PGL_2(q)\setminus T, \ y^{\varphi} = y^{-1} \}\cap G|\\
        &= 2(q-1)/(2,q-1).
    \end{aligned}
    \end{equation*}
    Hence $|G_{1,x}| > |G_{1, y}|$ if $m>2$. We therefore assume that $m = 2$ and $q\equiv 3\pmod 4$. Here Lemma~\ref{l:k=2_inv_L2}(iv) shows that $C_{\PGL_2(q)}(x) \cong D_{2(q+1)}$, so for an element $\varphi\in C_{\PGL_2(q)}(x)$, either $\varphi\in T$ and $\varphi\in G_{1,x}$, or $\varphi\in\PGL_2(q)\setminus T$ and $\varphi\sigma\in G_{1,x}$. Hence $|G_{1,x}| = 2(q+1) > |G_{1,y}|$.

    In Case (d), the group $G = \la T{:}L,\sigma\ra$ for some $L\leqs \Aut(T)$. Here  Lemma~\ref{l:k=2_stabiliser}(i) shows that $|G_{1, x}| = |I_L(x)|$, so it suffices to show that $|I_L(x)| > |I_L(y)|$. To see this, Lemma~\ref{l:k=2_inv_L2}(iv) implies that
    \begin{equation*}
        |I_L(x)| \geqs |I_{L_0}(x)| \geqs \ell(q+1) > \ell(q-1) = |I_L(y)|,
    \end{equation*}
    so the result follows.
\end{proof}

\subsection{Groups of Lie type in dimension at least three}

In this subsection we prove that if $T$ is a group of Lie type that is not isomorphic to $\LL_2(q)$ for any $q$,  then $\mathcal{G}(G) = 3$, hence proving Theorem~\ref{thm:greedy}(i) for these groups. Recall from Theorem~\ref{t:H_diag}(i) that $b(G) = 3$ for these groups, and that by Lemma~\ref{l:k=2_greedy_alt_spo} we may assume that $P = \Sn_2$.

We first verify Theorem~\ref{thm:greedy}(i) for some low-rank groups defined over small fields. Let $\mathcal{A}$ be the following collection of simple groups of Lie type: 
\begin{equation*}
\begin{aligned}
    \mathcal{A} = \{&\LL_3(q)\ (q\leqs 25),\ \LL_3(64),\ \LL_4(q)\ (q\leqs 17),\ \LL_5(2), \ \LL_6(2)\\
    &\UU_3(q)\ (q\leqs 32),\ \UU_4(q)\ (q\leqs 5),\ \UU_5(2), \ \UU_6(2),\\ 
    & \PSp_4(q)\ (3\leqs q\leqs 5),\ \PSp_6(2),\ \PSp_6(3), \ \PSp_8(2),\\ &  \POmega_7(3), \ \POmega_8^\pm(2), \ \POmega_8^+(3),\  \POmega^{-}_{10}(2),
    \ {^2}\mathrm{B}_2(8),\ {^2}\mathrm{B}_2(32), \ {^2}\mathrm{F}_4(2)',\ \mathrm{G}_2(3),\ {^3}\mathrm{D}_4(2)\}.
\end{aligned}
\end{equation*}

\begin{lem}
	\label{l:comp}
	If $T\in\mathcal{A}$ then $\mathcal{G}(G) = 3$.
\end{lem}

\begin{proof}
	For each group $T$, we use {\sc Magma} to compute the minimal value $v$ of
    $|I_{\Aut(T)}(y)|$ over all $y \in T$. 
	We then compute a set $S$ of $T$-class representatives of elements $x\in T$ satisfying the inequality $|I_{\Aut(T)}(x)| \leqs v \cdot |\Out(T)|$. By Lemma~\ref{l:k=2_greedy_alt_spo}, we may assume $P = \Sn_2$, so by Lemma~\ref{l:k=2_stabiliser}(iii) the set $S$ includes all $T$-conjugacy class representatives of all $x\in T$ with $|G_{1,x}|$ minimal. For each $x\in S$, we then check computationally that there exists $x_0\in T$ such that $I_{\Aut(T)}(x)\cap I_{\Aut(T)}(x_0) = 1$, and hence $\mathcal{G}(G) = 3$ by Lemma~\ref{l:k=2_stabiliser}(ii).
\end{proof}

The next result is our key tool for showing that every greedy base for the remaining groups $G$ of diagonal type with $k = 2$ has size three. 

\begin{lem}
    \label{l:k=2_probability}
    Let $\mathcal{P}$ be a set of $\Aut(T)$-class representatives of prime order elements in $\Aut(T)$, with $T$ non-abelian simple, and for each $y \in T$ let 
    \begin{equation*}
        \widetilde{Q}(T,y) = |I_{\Aut(T)}(y)|\cdot|\mathrm{Out}(T)|\cdot \sum_{z\in\mathcal{P}}\frac{|z^{\mathrm{Aut}(T)}\cap I_{\mathrm{Aut}(T)}(y)|}{|z^{\mathrm{Aut}(T)}|}.
    \end{equation*}
    If there exists an element $y \in T$ such that $\QQ(T, y) < 1$, then each primitive diagonal type group $G$ with socle $T^2$ satisfies $\mathcal{G}(G) = 3$.
\end{lem}

\begin{proof}
    Let $x\in T$ be such that $|G_{1,x}|$ is minimal amongst all two-point stabilisers in $G$. For $y\in T$, define $Q(T,x,y)$ to be the probability that a uniformly random conjugate $y^g$ of $y$ is such that
    \begin{equation*}
        I_{\Aut(T)}(x)\cap I_{\Aut(T)}(y^g) \ne 1.
    \end{equation*}
    By Lemma~\ref{l:k=2_greedy_alt_spo}, we may assume that $P = \Sn_2$, so Lemma~\ref{l:k=2_stabiliser}(iii) shows that $|I_{\Aut(T)}(x)| \leqs |I_{\Aut(T)}(y)|\cdot|\Out(T)|$. Hence we can use \cite[Lemma 2.1]{AB_reg} to bound
    \begin{equation*}
    \begin{aligned}
        Q(T,x,y) &\leqs \sum_{z\in\mathcal{P}}\frac{|z^{\Aut(T)}\cap I_{\Aut(T)}(x)|\cdot |z^{\Aut(T)}\cap I_{\Aut(T)}(y)|}{|z^{\Aut(T)}|}
        \\
        &\leqs |I_{\mathrm{Aut}(T)}(x)|\cdot \sum_{z\in\mathcal{P}}\frac{|z^{\mathrm{Aut}(T)}\cap I_{\mathrm{Aut}(T)}(y)|}{|z^{\mathrm{Aut}(T)}|}\\
        &\leqslant |I_{\Aut(T)}(y)|\cdot|\mathrm{Out}(T)|\cdot \sum_{z\in\mathcal{P}}\frac{|z^{\mathrm{Aut}(T)}\cap I_{\mathrm{Aut}(T)}(y)|}{|z^{\mathrm{Aut}(T)}|} = \widetilde{Q}(T,y).
    \end{aligned}
    \end{equation*}
    Thus if $\widetilde{Q}(T,y) < 1$ then $Q(T,x,y) < 1$. Finally,  Lemma~\ref{l:k=2_stabiliser}(ii) shows that if $Q(T,x,y) < 1$  for all such $x$ then $\mathcal{G}(G) \leqs 3$, 
    and hence $\mathcal{G}(G) = b(G) = 3$ by Theorem~\ref{t:H_diag}.
\end{proof}

Our notation for classical groups is as in \cite{BG_classical}: in particular for the orthogonal groups see \cite[Section 2.5]{BG_classical}. For example,  $\mathrm{O}^{\varepsilon}_n(q)$ denotes the group of isometries of a non-degenerate quadratic form (this group is called $\mathrm{GO}_n^\e(q)$ by some authors), and our $\mathrm{GO}_n^\e(q)$ is the group of similarities. 
By $\mathrm{Inndiag}(T)$ we denote the group of inner and diagonal automorphisms of $T$. More precisely, $\mathrm{Inndiag}(\LL_n^\e(q)) = \PGL_n^\e(q)$, $\mathrm{Inndiag}(\PSp_{2m}(q)) = \mathrm{PGSp}_{2m}(q)$, and if $T$ is orthogonal then $\mathrm{Inndiag}(T) = \mathrm{PSO}_n(q)$ if $n$ is odd, and an index two subgroup of the projective conformal group 
when $n$ is even. 

We now have a sequence of lemmas collecting data which we will eventually use to choose a suitable $y \in T$ for the remaining groups $T$ of Lie type, and then to bound $\QQ(T, y)$.

\begin{lem}[{\cite[Proposition~3.9]{B_fpr2}}]
	\label{l:order}
	The following inequalities hold for all $n,m\geqs 1$ and all prime powers $q$. 
	\begin{itemize}\addtolength{\itemsep}{0.2\baselineskip}
		\item[{\rm (i)}] $\frac{1}{2}q^{n^2-1} < |\PGL_n^\e(q)| < q^{n^2-1}$, $|\GL_n(q)| < q^{n^2}$ and $|\GU_n(q)| \leqs (q+1)q^{n^2-1}$.
		\item[{\rm (ii)}] $\frac{1}{2}q^{m(2m+1)} <  |\Sp_{2m}(q)| = |\mathrm{SO}_{2m+1}(q)| < q^{m(2m+1)}$.
		\item[{\rm (iii)}] $\frac{1}{2}q^{m(2m-1)}< |\mathrm{Inndiag}(\POmega_{2m}^\e(q))| = |\mathrm{SO}_{2m}^\e(q)| < q^{m(2m-1)}$.
		\item[{\rm (iv)}] If $q$ is odd then $|\mathrm{Inndiag}(\POmega_{2m}^\e(q))| = |\mathrm{PO}_{2m}^\e(q)|$.
	\end{itemize}
\end{lem}

In the next two results, we shall use some ideas and notation from \cite[Section 3.2.1]{BG_classical}. 
For a prime $r$ and prime power $q$ such that $r \nmid q$, we write $\Phi(r, q)$ for the minimal $i \in \mathbb{Z}_{>0}$ such that $r \mid q^i-1$.
Each element $x \in \GL_n(q)$ of prime order $r$, with $(r, q) = 1$ and $i := \Phi(r, q)$, fixes a direct sum decomposition $V = U_1 \oplus \cdots \oplus U_s \oplus C_x$, where each $U_j$ is an $i$-dimensional space on which $x$ acts irreducibly, and $x$ acts trivially on $C_x$. Over $\bbF_{q^i}$ the element $x|_{U_j}$ has $i$ distinct eigenvalues, and we write
$$x = [\Lambda_1^{a_1}, \ldots, \Lambda_t^{a_t}, I_e]$$ to denote that over $\bbF_{q^i}$ the set $\Lambda_j$ of $i$ eigenvalues occurs with multiplicity $a_i$, so that $n = i\sum_{j}a_j + e$. This determines $x$ up to $\GL_n(q)$-conjugacy.

\begin{lem}\label{l:class_uni} 
Suppose that $T$ is a classical group that has natural module $V$ of dimension at least three, is not in $\mathcal{A}$ and is not equal to $\POmega_{2m}^+(q)$ with $m\geqs 4$.  Let $\ell = 1$ if  $T  \in \{\UU_{2m}(q), \Omega_{2m+1}(q)\}$, and $\ell = 0$ otherwise; and let $\alpha = q-1$ if $T = \LL_n(q)$, and $\alpha = q+1$ otherwise. 
Let $V_2$ be a non-degenerate $\ell$-space of $V$, and write $V = V_1 \perp V_2$, with $V_1$ of minus type when $T = \Omega_{2m+1}(q)$.
Then $\Aut(T) \cap \PGL(V)$ contains an element $\widehat{y}$ of order $c$, as given in 
Table~\ref{tab:class_uni}, such that
$\la \widehat{y} \ra$ acts irreducibly on $V_1$ and trivially on $V_2$.

Let $y$ be a generator of $\la \widehat{y} \ra \cap T$.
Then all of the following properties hold.
    \begin{itemize}\addtolength{\itemsep}{0.2\baselineskip}
		\item[{\rm (i)}] $C_{\Aut(T)}(y) = \la \widehat{y} \ra$ is cyclic of order $c$.
		\item[{\rm (ii)}] Each $t\in\Aut(T)$ that inverts $y$ is $\Aut(T)$-conjugate to an involution $t_0$ whose centraliser is described in Column 3, and there are at most $c$ such involutions $t$. 
        \item[{\rm (iii)}] The value of $b_0$ in Column 4 is a  lower bound for $|t_0^{\Aut(T)}|$ (we combine various cases for $q$ and $n$). 
    \end{itemize}
\end{lem}

{\small
	\centering
	\begin{table}
		\caption{Triples $(T,y,t_0)$ with $T\notin\mathcal{A}$ a classical group, $y \in T$ and $t_0 \in I_{\Aut(T)}(y) \setminus C_{\Aut(T)}(y)$}
		\begin{tabular}{@{}llll@{}}
			\toprule
			$T$ & $c$  & $C_T(t_0)$ & $b_0$ \\ \midrule
            $\LL_3^\e(q)$ & $(q^n-\e)/\alpha$ & $\PGL_2(q)$ 
            & $q^2(q^3-\e)$ \\ \hline
			$\LL_n(q)$,  & $(q^n-1)/\alpha$ & $\mathrm{PSO}_n^\epsilon(q)$, $q$ or $n$ odd & $\frac{1}{2\alpha}q^{\frac{n^2}{2}+\frac{n}{2}-1}$\\
			$n\geqs 4$ && $[q^{n-1}].\Sp_{n-2}(q)$, $q$ and $n$ even & \\
            \hline
            $\UU_4(q)$ & $q^3 + 1$ & $[q^3].\Sp_2(q)$, $q$ even & $\frac{1}{4}q^2(q^3+1)(q^4-1)$
            \\
            & & $\mathrm{PSO}_4^\epsilon(q)$, $q$ odd & \\
            \hline 
			$\UU_n(q)$,  & $(q^n+1)/\alpha$ & $\mathrm{PSO}_n(q)$ & $\frac{1}{2\alpha}q^{\frac{n^2}{2} + \frac{n}{2}-1}$\\
            $n\geqs 5$ odd & \\
            \hline 
 			$\UU_n(q)$, & $q^{n-1}+1$ & $\mathrm{PSO}_n^\epsilon(q)$, $q$ odd & $\frac{1}{2\alpha}q^{\frac{n^2}{2} + \frac{n}{2}-1}$\\
			$n \geqs 6 $ even & & $[q^{n-1}].\Sp_{n-2}(q)$, $q$ even 
            \\
            \hline 
            $\PSp_4(q)$ & $q^2 + 1$ & $(\GU_2(q)/2).2$, $q \equiv 3 \pmod 4$ & $\frac{1}{2}q^3(q^2 + 1)(q - 1)$\\
            & & $(\GL_2(q)/2).2$, $q \equiv 1 \pmod 4$ & \\
            && $[q^4]$, $q$ even &\\
            \hline 
			$\PSp_{2m}(q)$ & $q^m+1$ & $(\GU_m(q)/2).2$, $q\equiv 3 \pmod 4$ & $\frac{1}{4\alpha}q^{m^2+m+1}$\\
            $m \geqs 3$ && $(\GL_m(q)/2).2$, $q \equiv 1 \pmod 4$ &\\
			&& $[q^{\frac{1}{2}(m^2+3m-2)}].\Sp_{m-2}(q)$, $q$ and $m$ even&\\
			&& $[q^{\frac{1}{2}(m^2+m)}].\Sp_{m-1}(q)$, $q$ even, $m$ odd&\\
            \hline 
			$\Omega_{2m+1}(q)$ & $q^m+1$ & $(\Omega_m^\epsilon(q)\times \Omega_{m+1}^{\epsilon'}(q)).2^2$ & $\frac{1}{4}q^{m^2+m}$\\
			\hline 
			$\POmega_{2m}^-(q)$ & $q^m+1$ & $(\Omega_m^\epsilon(q)\times\Omega_m^{-\epsilon}(q)).2^2$, $q$ odd & $\frac{1}{8}q^{m^2}$\\
			&& $[q^{\frac{1}{2}(m^2+m-2)}].\Sp_{m-2}(q)$, $q$ and $m$ even\\
			&& $[q^{\frac{1}{2}(m^2-m)}].\Sp_{m-1}(q)$, $q$ even, $m$ odd\\
			\bottomrule
		\end{tabular}
		\label{tab:class_uni}
	\end{table}
}

Unfortunately our main source \cite{LL_chiral} for Table~\ref{tab:class_uni} has an error in the lines for $\POmega^+_{2m}(q)$ (the centraliser contains the direct product $C_{q+1} \times C_{(q^{m-1}+1)/(q-1, 2)}$), so we shall consider those groups separately.

\begin{proof} 
If $T = \LL_3^\e(q)$, then the existence of $\wy$ is clear. 
Then 
$C_{\Aut(T)}(y) = \la \wy \ra$, and there is an involutory graph automorphism $t_0$ of $T$ that inverts $y$. In the notation of \cite[Sections 3.2.5 and 3.3.5]{BG_classical}, this involution $t_0$ is of type $\gamma_1$, and so has a centraliser in $\PGL^\e_3(q)$ of order $|\Sp_2(q)|$. 
Hence $|t_0^{\Aut(T)}| \geqs |t_0^{\PGL^{\e}_3(q)}| = |\PGL^\e_3(q)|/|\PGL_2(q)|$ 
is as given.

For $T\ne \LL_3^\e(q)$, the second and third column of Table~\ref{tab:class_uni} are essentially \cite[Table~4]{LL_chiral}. Since $\wy$ and $y$ act irreducibly on $V_1$, the centraliser in $\PGL(V_1)$ of $y|_{V_1}$ is a Singer cycle on $V_1$, and hence is cyclic. Since $\dim(V_2) \leqs 1$ it follows that the centraliser in $\Aut(T) \cap \PGL(V)$ of $y$ is cyclic. Furthermore, field automorphisms introduce no new centralising elements, so $C_{\Aut(T)}(y) = \la \wy \ra$ is cyclic.

For Part (ii), it is noted in \cite[p. 587]{LL_chiral} that all involutions in $\Aut(G)$ that invert $y$ are conjugate to $t_0$. Furthermore, in all cases $|I_{\Aut(T)}(y)| = 2c$, and in particular there are at most $c$ involutions which invert $y$, since these  lie in $I_{\Aut(T)}(y)\setminus C_{\Aut(T)}(y)$. 

For Part (iii), we note in all cases that $|t_0^{\Aut(T)}| \geqs |t_0^T| = |T:C_T(t_0)|$. 
Suppose first that $T = \LL_n^\e(q)$ with $n \geqs 4$. 
If $T \ne \UU_4(q)$ then Lemma~\ref{l:order} shows that $|C_T(t_0)| < q^{n(n-1)/2}$, whilst $|T| \geqs \frac{1}{2\alpha}q^{n^2-1}$, so the result follows.
If instead $T = \UU_4(q)$ then $|C_T(t_0)| = q^4(q^2-1)$ if $q$ is even and $\frac{1}{2}q^2(q^4-1)$ otherwise, which yields $|C_T(t_0)|\leqs q^4(q^2-1)$ in all cases. The result follows from $|T| \geqs  \frac{1}{4}q^6(q^2-1)(q^3+1)(q^4-1)$. 

Suppose next that $T = \PSp_{2m}(q)$. If $m = 2$ then $|C_T(t_0)| \leqs   (q + 1)q(q^2-1)$ if $q$ is odd, whilst $|C_T(t_0)| = q^4$ if $q$ is even, so
in all cases $|C_T(t_0)| \leqs (q+1)q(q^2-1)$. Furthermore $|\PSp_4(q)| \geqs\frac{1}{2}q^4(q^2-1)(q^4-1)$, so the result for $m = 2$ follows.
If $m\geqs 3$ then one can check using Lemma~\ref{l:order} that $|C_T(t_0)| < (q+1)q^{m^2-1}$ in all cases, whilst $|T| \geqs \frac{1}{2}|\Sp_{2m}(q)| \geqs\frac{1}{4}q^{m(2m+1)}$, so the result follows. 

We turn to the orthogonal groups. Lemma~\ref{l:order}(ii),(iii) shows that for all $n$ we can bound $|\Omega_{n}^\epsilon(q)| = \frac{1}{2}|\mathrm{SO}_{n}^\epsilon(q)| < \frac{1}{2}q^{n(n-1)/2}$. Thus, if $T = \Omega_{2m+1}(q)$ then $|C_T(t_0)| \leqs q^{m^2}$, whilst $|T| \geqs \frac{1}{4}q^{m(2m+1)}$ by Lemma~\ref{l:order}(ii), giving the required bound.
Finally, suppose $T = \POmega_{2m}^-(q)$. 
If $q$ is even, then $|C_T(t_0)| < q^{m^2-m}$ by 
Lemma~\ref{l:order}(ii), and arguing as in the previous case shows that $|C_T(t_0)| < q^{m^2-m}$ also holds if $q$ is odd. Noting that $|T| \geqs \frac{1}{4}|\mathrm{PO}^-_{2m}(q)| > \frac{1}{8}q^{m(2m-1)}$ gives the result. 
\end{proof}

{\small
	\centering
	\begin{table}
		\caption{Elements $z\in C_{\Aut(T)}(y)$ of prime order $r$ and their centralisers}
		\begin{tabular}{@{}lllll@{}}
			\toprule
			$T$ & $|C_L(z)|$  & $|z^{\Aut(T)}|\geqs$ & Conditions\\ \midrule
            $\LL_3(q)$ & $|\GL_1(q^3)|/(q-1)$ & $q^3(q^2-1)(q-1)$ & $i > 1$ \\
            & $3|\GL_1(q^3)|/(q-1)$ & $q^3(q^2-1)(q-1)/3$ & $i = 1$  \\
            \hline
			$\LL_n(q)$, & $|\GL_{n/i}(q^i)|/(q-1)$ & $\frac{1}{2}(q-1)q^{n^2/2-1}$ & $i > 1$  \\
            $n\geqs 4$  & $r|\GL_{n/r}(q^r)|/(q-1)$ & 
            $\frac{1}{4}(q-1)q^{n^2/2-1}$ & $i = 1$ \\
            \hline
            $\UU_3(q)$ & $|\GU_1(q^3)|/(q+1)$ & $q^3(q^2-1)(q+1)$ & $i > 2$\\
            & $3|\GU_1(q^3)|/(q+1)$ & $q^3(q^2-1)(q+1)/3$ & $i \leqs 2$\\
            \hline
            $\UU_n(q)$, &  $|\GU_{2n/i}(q^{i/2})|/(q+1)$ & $\frac{1}{2}q^{2n^2/3}$ & $i > 2$ \\
            $n\geqs 5$ odd &  $r|\GU_{n/r}(q^r)|/(q+1)$ & $\frac{1}{6}q^{2n^2/3}$ & $i \leqs 2$ \\
            \hline
            $\UU_n(q)$,  
            & $|\GU_{2(n-1)/i}(q^{i/2})|$ & $\frac{1}{4}q^{2(n^2+n-4)/3}$ & $i > 2$\\
            $n$ even  & $|\GU_{n-1}(q)|$ & $q^{n-1}(q^n-1)/(q+1)$ & $i \leqs 2$\\
            \hline
            $\PSp_{4}(q)$ & $|\GU_{4/i}(q^{i/2})|$ & $q^3(q^2+1)(q-1)$ & $i > 1$ \\
            & $2|\GU_m(q)|$ & $\frac{1}{2}q^3(q^2+1)(q-1)$ & $i = 1$ \\
                        \hline
            $\PSp_{2m}(q)$ & $|\GU_{2m/i}(q^{i/2})|$ & $\frac{1}{2}q^{m^2+m+1}/(q+1)$ & $i > 1$ \\
            $m \geqs 3$ & $2|\GU_m(q)|$ & $\frac{1}{4}q^{m^2+m+1}/(q+1)$ & $i = 1$ \\
            \hline
            $\Omega_{2m+1}(q)$ & $|\GU_{2m/i}(q^{i/2})|$ & $\frac{1}{2}q^{m^2+m+1}/(q+1)$ & $i > 1$ \\
            & $2|\mathrm{SO}_{2m}^-(q)|$ & $\frac{1}{2}q^m(q^m-1)$ & $i = 1$ \\
            \hline 
            $\POmega_{2m}^-(q)$  & $|\GU_{2m/i}(q^{i/2})|$ & $\frac{1}{2}q^{m^2-m+1}/(q+1)$ & $i > 1$ \\
            & $|\GU_m(q)|$ & $\frac{1}{2}q^{m^2-m+1}/(q+1)$ & $i = 1$, $m$ odd \\
            & $2|\mathrm{SO}_m^-(q^2)|$ & $\frac{1}{2}q^{m^2-m+1}/(q+1)$ & $i = 1$, $m$ even \\
			\bottomrule
		\end{tabular}
		\label{tab:class_centr}
	\end{table}
}

With $c$ as in Lemma~\ref{l:class_uni}, notice that if $r \mid c$ is  prime then 
$r \nmid q$. 
 
\begin{lem}\label{l:z_elts}
    Suppose $T$ is a classical group with natural module $V$ of dimension $n \geqs 3$,
    such that $T \not\in \mathcal{A}$ and $T \ne \POmega_{2m}^+(q)$ with $m\geqs 4$. Let $L = \mathrm{Inndiag}(T)$, let $y \in T$ be as described in Lemma~\ref{l:class_uni}, and let $z \in C_{\Aut(T)}(y)$ have prime order $r$. 
       
    
    Then $|C_L(z)|$ and a lower bound on $|z^{\Aut(T)}|$ are as described in Table~\ref{tab:class_centr}, where $i = \Phi(r,q)$. Moreover, $z$ is $\Aut(T)$-conjugate to the involution $t_0$ from Lemma~\ref{l:class_uni} if and only if $T = \PSp_{2m}(q)$ with $q\equiv 3\pmod 4$.
\end{lem}

\begin{proof}
    Let $\beta = 2$ 
    if $T$ is unitary, and $\beta = 1$ otherwise, and let $\ell \in \{0, 1\}$ be as in Lemma~\ref{l:class_uni}.
    Throughout the proof, we shall make extensive use of the fact that $y$ and $\wy$ act irreducibly on the subspace $V_{1}$ from Lemma~\ref{l:class_uni}, which is
    of dimension $n-\ell$ over $\bbF_{q^\beta}$, and is non-degenerate if $T \neq \LL_n(q)$. 
    Hence from $z \in C_{\Aut(T)}(y) = \la \wy \ra$ we deduce that $z|_{V_{1}}$ is similar to a block diagonal matrix with all blocks identical, and so 
    $i \mid \beta(n-\ell)$. By a standard calculation (see for example \cite[Lemma A.1]{BG_classical}) if $r$ is an odd prime divisor of $q^m+1$ then $\Phi(r, q)$ is even, and  in this case $\Phi(r, q) \equiv 0 \pmod 4$ if and only if $m$ is even.
    In each case we shall first find $|C_{L}(z)|$, then use $|z^{\Aut(T)}| \geqs |z^L| = |L:C_L(z)|$. 

    Suppose $T = \LL_n(q)$, so $L = \PGL_n(q)$, $\ell = 0$ and $r$ divides $(q^n-1)/(q-1)$. Consulting \cite[Tables~B.1 and B.3]{BG_classical}, we see that if $i = 1$ then $r \mid n$, and in all cases $z$ is of type $[\Lambda^{n/r}]$ or $t_{n/2}'$, and  $|C_{L}(z)|$ is as given in Table~\ref{tab:class_centr}.  
    For $n = 3$ we easily get the given bound on $|z^{\Aut(T)}| \geqs |L:C_L(z)|$. 
    For $n \geqs 4$ we first use Lemma~\ref{l:order}(i) to bound $|L| \geqs \frac{1}{2}q^{n^2-1}$. Next, if $i > 1$ then $|C_L(z)| \leqs (q-1)^{-1}q^{n^2/i}$, whilst if $i = 1$ then $|C_L(z)|  \leqs r(q-1)^{-1}q^{n^2/r}$. The bound on $|z^{\Aut(T)}| \geqs |L:C_L(z)|$ follows in both cases.
    For $r = 2$, comparing centralisers with Lemma~\ref{l:class_uni} shows that $z$ is not conjugate to $t_0$.

    Our next case is $T = \UU_n(q)$, so $L = \PGU_n(q)$. Suppose first that $n$ is odd, so that $\ell = 0$. Then $|\wy| = (q^n + 1)/(q + 1)$ is odd, and hence $r$ is an odd  divisor of $q^n+1$, and $i \equiv 2 \pmod 4$.  We consult \cite[Table B.4]{BG_classical} and see that 
    if $i > 2$  then $z$ is of the form $[\Lambda^{2n/i}]$, whilst if $i = 2$ then $r$ divides $(q+1, n)$ and $z$ is of the form $[\Lambda^{n/r}]$. In both cases we find $|C_{L}(z)|$ as given, and use Lemma~\ref{l:order}(i) to bound
    \[ |L:C_L(z)| \geqs  \begin{cases}
        \frac{q^{n^2-1}/2}{q^{2n^2/i - 1}} \geqs \frac{1}{2}q^{\frac{2}{3}n^2} & \mbox{ if $i > 2$, so that $i \geqs 6$}\\
        \frac{q^{n^2 - 1}/2}{rq^{n^2/r - 1}} = \frac{1}{2r}q^{n^2(1 - 1/r)} \geqs \frac{1}{6}q^{\frac{2}{3}n^2} & \mbox{ if $i = 2$, so $r \mid (q+1)$.}\\
    \end{cases}
      \]  
      
    Suppose instead that $n$ is even, so 
    $\dim(V_1) = n-1$ and $|\wy| = q^{n-1}+1$, so that again if $r$ is odd then $i \equiv 2 \pmod 4$. 
    From the irreducibility of $\la \wy \ra$ on $V_{1}$ we deduce from \cite[Table~B.4]{BG_classical} that 
    if $i > 2$ then 
    $z$ is of the form $[\Lambda^{2(n-1)/i}, I_1]$ with $|C_{L}(z)|$ as given. Hence if $i > 2$ then we can bound $|L| \geqs \frac{1}{2}q^{n^2-1}$ and
    \begin{equation*}
    |C_L(z)| \leqs (q^{i/2}+1)q^{2(n-1)^2/i-i/2} \leqs 2q^{2(n-1)^2/i} \leqs 2q^{(n-1)^2/3}
    \end{equation*}
    to bound $|z^{\Aut(T)}|$. 
    If $i = 2$ or $r = 2$ (so $i = 1$) then $r \mid (q+1)$, the element $z$ is of the form $t_1$ 
   or $[I_1, \lambda I_{n-1}]$ for some primitive $r$th root of unity $\lambda$, 
    and $|C_L(z)|$ is as given. 
    Here we use the exact value for $|L| = |\PGU_n(q)|$ and $|C_L(z)| = |\GU_{n-1}(q)|$ to bound $|z^{\Aut(T)}|$. 
Again, comparing centralisers with Lemma~\ref{l:class_uni} shows that if $r = 2$ then $z$ is not conjugate to $t_0$.

Our next case is $T = \PSp_{2m}(q)$, for which $L = \mathrm{PGSp}_{2m}(q)$, the group $\la y \ra$ is again irreducible, and $r \mid q^{m}+1$, so that if $r$ is odd then again $i$ is even. In particular,  $i = 1$ if and only if $r = 2$, which happens if and only if $q$ is odd. If $i > 1$ then we see from \cite[Table B.7]{BG_classical} that $z$ is of the form $[\Lambda_s^{2m/i}]$, with $|C_L(z)|$ as given. Parts (i) and (ii) of Lemma~\ref{l:order} yield
	\begin{equation*}
	|L:C_L(z)| = \frac{|\mathrm{PGSp}_{2m}(q)|}{|\GU_{2m/i}(q^{i/2})|} \geqs \frac{|\mathrm{PGSp}_{2m}(q)|}{|\GU_{m}(q)|} \geqs \frac{1}{2}(q+1)^{-1}q^{m^2+m+1},
	\end{equation*}
    with a more precise calculation giving the value when $m = 2$.
If $i = 1$  and $r = 2$, then $z$ is an involution of type $t_{m}'$ (see \cite[Section 3.4.2.5]{BG_classical}), and so $|C_L(z)|$ and an upper bound on $|z^{L}|$ are as given.  Comparing centralisers, and looking in Table \cite[Table B.7]{BG_classical} at all other involution centralisers, we see that $z$ is $\Aut(T)$-conjugate to $t_0$ if and only if $q \equiv 3 \pmod 4$. 

Our final case is $T = \POmega_{n}^\varepsilon(q)$, with $n \in \{2m, 2m+1\}$ and $\epsilon \in \{\circ, -\}$.
Here $\la y \ra$ acts irreducibly on a non-degenerate $2m$-space $V_1$, and on its orthogonal complement, and $|\wy| = q^{m}+1$. Hence $i = 1$ if and only if $r = 2$, which happens if and only if $q$ is odd, and otherwise $i$ is even. 
If $i > 1$ then we see from \cite[Table~B.12]{BG_classical} that $z$ is of the form $[\Lambda^{(2m-\ell)/i},I_{\ell}]$, and hence $|C_{L}(z)|$ is as given, with the bound on $|z^{\Aut(T)}| \geqs |L:C_L(z)|$ following from an easy calculation.
If $i = 1$ (so $r = 2$) and $n = 2m+1$ then $z$ is an involution of type $t_m'$ (see \cite[Section 3.5.2.2]{BG_classical}), whence 
\begin{equation*}
|z^{\Aut(T)}|\geqs |L:C_L(z)| = \frac{|\mathrm{SO}_{2m+1}(q)|}{2|\mathrm{SO}_{2m}^-(q)|} = \frac{1}{2}q^m(q^m-1).
\end{equation*}
Finally, assume $i = 1$ (so $r = 2$) and $\varepsilon = -$. If $m$ is even then $z$ is an involution of type $t_{m/2}'$, whereas if $m$ is odd then $z$ is of type $t_m$ (see \cite[Sections 3.5.2.9 and 3.5.2.13]{BG_classical}). In both cases, Lemma~\ref{l:order} shows that $|C_L(z)| \le |\GU_{m}(q)| \le (q+1) q^{m^2-1}$, whilst $|L| \geqs \frac{1}{2}q^{m(2m-1)}$, so the result follows. 
       \end{proof}

\begin{cor}\label{cor:q_bound}
     Suppose that $T$ is a classical group with natural module $V$ of dimension $n \geqs 3$,
    such that $T \not\in \mathcal{A}$ and $T \ne \POmega_{2m}^+(q)$ with $m\geqs 4$.  Let $y \in T$, $t_0 \in \Aut(T)$ and $c \in \mathbb{Z}$ be as in Lemma~\ref{l:class_uni}, and let $C = C_{\Aut(T)}(y)$. Let $\beta = 2$ if $T = \UU_n(q)$, and $\beta = 1$ otherwise, and fix $a \ge  (q^\beta-1,c)$. 
    
    Let $b_0, b_1, b_2\in \mathbb{R}_{>0}$ be such that $|t_0^{\Aut(T)}| \ge b_0$, $|z^{\Aut(T)}| \ge b_1$ for all $z \in C$ of prime order with $\Phi(|z|,q) > \beta$, and $|z^{\Aut(T)}| \ge b_2$ for all other $z \in C$ of prime order. 
    If, for some choice of  $a, b_0, b_1, b_2$,
    \begin{equation*}
    2c\cdot |\Out(T)|\cdot \left(\frac{c}{b_0} + \frac{c}{b_1} + \frac{a-1}{b_2}\right) < 1,
    \end{equation*}
    then $\mathcal{G}(G) = 3$.
\end{cor}

\begin{proof}
Let $I = I_{\Aut(T)}(y)$. We seek to bound the expression
\begin{equation}\label{eq:prob_again}
\widetilde{Q}(T,y) = |I|\cdot|\mathrm{Out}(T)|\cdot \sum_{z\in\mathcal{P}}\frac{|z^{\mathrm{Aut}(T)}\cap I|}{|z^{\mathrm{Aut}(T)}|},
\end{equation}
from Lemma~\ref{l:k=2_probability}. To do so we will divide the elements $z$ of prime order in $I$ into three sets: those for which  $\Phi(|z|, q) > \beta$, the remaining elements of prime order in $I$ 
that are not conjugate to $t_0$, and the involutions conjugate to $t_0$. By Lemma~\ref{l:class_uni} the group $I$ has shape $c{:}2$.

Firstly, every prime order $z\in I$ with $\Phi(|z|,q) > \beta$ has odd order, and hence lies in $C$. Therefore there are at most $|C| = c$ prime order elements $z\in I$ with $\Phi(|z|,q) > \beta$. Secondly, every prime order element of $I \setminus C$  is an involution, and by Lemma~\ref{l:class_uni} is conjugate to $t_0$.  

If $T \neq \PSp_{2m}(q)$ with $q \equiv 3 \pmod 4$, then Lemma~\ref{l:z_elts} shows that no element of $C$ is $\Aut(T)$-conjugate to $t_0$, so there are at most $c$ elements in $I$ that are $\Aut(T)$-conjugate to $t_0$. Each  $z\in C$ of prime order such that $\Phi(|z|,q) \leqs \beta$ has order dividing $q^\beta -1$, and so lies in the unique subgroup of $C$ of order $(q^\beta-1,c) \leqs a$. Thus, there are at most $a-1$ prime order elements in $I$ for which $\Phi(|z|,q) \leqs \beta$ and $z \not\sim_{\Aut(T)} t_0$.
Substituting in the three cases of $z \in I \setminus C$, $z\in C$ with $\Phi(|z|,q) > \beta$, and $z\in C$ with $\Phi(|z|,q) \leqs \beta$, into \eqref{eq:prob_again} gives an upper bound on $\widetilde{Q}(T,y)$. The result then follows from Lemma~\ref{l:k=2_probability}.

If $T = \PSp_{2m}(q)$ with $q\equiv 3\pmod 4$, then $(q^\beta-1,c) = (q-1, q^m+1) = 2 \leqs a$. Hence there is a unique $z\in C$ of prime order with $\Phi(|z|,q) \leqs \beta$, and this $z$ is conjugate to $t_0$ by Lemma~\ref{l:z_elts}: hence $b_2 \le |t_0^{\Aut(T)}|$ in this case.
Thus there are at most $c+1$ elements in $I$ that are conjugate to $t_0$, so 
\begin{equation*}
\frac{|t_0^{\Aut(T)}\cap I|}{|t_0^{\Aut(T)}|} \leqs \frac{c+1}{|t_0^{\Aut(T)}|} \leqs \frac{c}{b_0} + \frac{a-1}{b_2}.
\end{equation*}
The result follows from \eqref{eq:prob_again}.
%
%
%
\end{proof}

\begin{lem}
	\label{l:luso}
	Suppose $T = \LL_n(q)$ or $\UU_n(q)$ with $n\geqs 3$, or $T = \PSp_{2m}(q)$ with $2m \ge 4$, or $T = \POmega^\e_{n}(q)$ with $n\ge 7$ and $\e \in \{\circ, -\}$. Then $\mathcal{G}(G) = 3$.
\end{lem}

\begin{proof}  
If $T \in \mathcal{A}$ then the result is immediate from Lemma~\ref{l:comp}, so we may assume $T\notin \mathcal{A}$ and let $y$, $c$ and $b_0 \leqs |t_0^{\Aut(T)}|$ be as in Lemma~\ref{l:class_uni}.  Let $\beta =2$ if $T$ is unitary and $\beta = 1$ otherwise. Let $b_1$ and $b_2$ be the lower bounds on $|z^{\Aut(T)}|$ in Lemma~\ref{l:z_elts} for $z \in C_{\Aut(T)}(y)$ of prime order, with $\Phi(|z|,q) > \beta$ and $\Phi(|z|,q) \le \beta$, respectively.

If $T = \LL_n(q)$  then $c =(q^n-1)/(q-1)$, so  we may set $a = q-1 \geqs (q-1, c)$. We define
\[ \omega = 
\begin{cases}
   |\Out(T)| & \mbox{if $n = 3$ and $q \leqs 73$}\\
   6\log q & \mbox{if $n = 3$ and $q > 73$}\\
   8\log q & \mbox{if $n = 4$}\\
    2(q-1)\log q & \mbox{if $n \geqs 5$.} 
\end{cases}
\]

If $T = \UU_n(q)$ then $c = q^{n-1} + 1$ if $n$ is even, whilst  $c = (q^n+1)/(q+1)$ if $n$ is odd, so we may let $a = q+1 \geqs (q^2-1, c)$. We define 
\[ \omega = 
\begin{cases}
   |\Out(T)| & \mbox{if $n = 3$ and $q \leqs 73$, or $n = 4$ and $q\in\{7,8\}$}\\
   6\log q & \mbox{if $n = 3$ and $q > 73$}\\
   8\log q & \mbox{if $n = 4$ and $q \geqs 9$}\\
   2(q+1)\log q & \mbox{if $n \geqs 5$.} 
\end{cases}
\]

If $T = \PSp_{2m}(q)$ then $c = q^m+1$, and we may let $a = 2 \ge (q-1, c)$. 
We define $\omega = 2\log q$.

If $T = \POmega^\e_n(q)$ then $c = q^m+1$, so we may set $a = (2, q-1) = (c, q-1)$. If $T = 
\Omega_{2m+1}(q)$ we let $\omega = 2\log q$, and if $T = \POmega_{2m}^-(q)$ we let $\omega = 8 \log q$.
    
Then in all cases  $|\Out(T)| \leqs \omega$. Our assumption that $T \not\in \mathcal{A}$ eliminates many small values of $n$ and $q$, so
the result now follows from Corollary~\ref{cor:q_bound}.
\end{proof}

As mentioned before the proof of Lemma~\ref{l:class_uni},
for the orthogonal groups of plus type we work from scratch.

\begin{lem}
	\label{l:o+}
	Suppose $T = \POmega_{2m}^+(q)$ with $m \ge 4$. Then $\mathcal{G}(G) = 3$.
\end{lem}

\begin{proof}
	Let $V$ be the natural module for $\mathrm{O}^+_{2m}(q)$, and let $A \in \GL_m(q)$ be of order $q^m-1$. Then as noted in \cite[Lemma 1.2.2]{Kleidman} the group $\mathrm{O}^+_{2m}(q)$ contains an element $\widetilde{y}$ which acts as $A$ on a totally singular $m$-space $E \le V$, and as $A^{-T}$ (inverse-transpose) on a complementary totally singular $m$-space $F \le V$. 
    Furthermore, if $q$ is even then $\widetilde{y} \in \Omega_{2m}^+(q)$, whilst if $q$ is odd then $\widetilde{y} \in \mathrm{SO}^+_{2m}(q) \setminus \Omega^+_{2m}(q)$.
    Let $\widehat{y}$ be the image of $\widetilde{y}$ in $\mathrm{PSO}_{2m}^+(q)$, and note that $\la \widetilde{y} \ra \cap Z(\mathrm{SO}_{2m}^+(q)) = \la -I\ra$, 
    so $|\widehat{y}| = 
    (q^m-1)/(2,q-1)$. 
    
    Let $y\in T$ be a generator of $\la \widehat{y} \ra \cap T$, so that $y = \widehat{y}^2$ if $q^m \equiv 1 \pmod 4$ and $y = \widehat{y}$ otherwise. 
    We now determine $|C_{\Aut(T)}(y)|$. 
    For each primitive element $\lambda$ of $\bbF_{q^m}$, let $\mathcal{E}_{\lambda} = \{\lambda, \lambda^{q}, \lambda^{q^{2}}, \ldots, \lambda^{q^{m-1}}\}$. Then over $\bbF_{q^m}$, for some choice of $\lambda$ the element $\widetilde{y}$ has eigenvalues $\mathcal{E}_{\lambda} \cup \mathcal{E}_{\lambda}^{-1}$. The only scalar multiples of $\widetilde{y}$ lying in $\mathrm{O}^+_{2m}(q)$ are $\widetilde{y}$ and (only if $q$ is odd) $-\widetilde{y} = \widetilde{y}^{(q^m-1)/2+1}$ 
    and a short calculation shows that
    \begin{equation*}
    \{x: x \in \mathcal{E}\} \cap \{\pm x : x \in \mathcal{E}^{-1}\} = \emptyset = \{x^2 : x \in \mathcal{E}\} \cap \{\pm x^{2} : x \in \mathcal{E}^{-1}\}.
    \end{equation*}
    Hence (see also the proof of \cite[Lemma 4.1.9]{KL_classical}) 
    $C_{\mathrm{O}^+_{2m}(q)}(\widetilde{y}) = C_{\GL_{2m}(q)}(\widetilde{y}) \cap \mathrm{O}^+_{2m}(q) = \la \widetilde{y} \ra$, and 
    \[ C_{\mathrm{PO}^+_{2m}(q)}(\wy) = C_{\mathrm{O}^+_{2m}(q)} (\widetilde{y})/\langle -I \rangle = \la \wy \ra \leqs \mathrm{PSO}^+_{2m}(q).\] 
    There exists an element of $\mathrm{GO}^+_{2m}(q) \setminus \mathrm{O}^+_{2m}(q)$ which acts as a primitive scalar on $E$ and trivially on $F$ (see \cite[p. 35]{BHRD}), and hence centralises $\widetilde{y}$. Hence the centraliser in $\mathrm{PGO}_{2m}^+(q)$ of $y$ 
    has order $(2, q-1)|\wy| = q^m-1$.

    It is clear that $C_{\mathrm{P\Gamma O}_{2m}^+(q)}(y) \leqs \mathrm{PGO}^+_{2m}(q)$, so to determine $|C_{\Aut(T)}(y)|$, all that remains is to consider the additional graph automorphisms when $m = 4$. By \cite[Prop 1.4.1]{Kleidman} no triality automorphism $\tau$ centralises 
    an element of order $(q^4-1)/(2, q-1)^2$, and consulting \cite[Table 4.5.1]{GLS_CFSG3} we see that the centralisers in $T$ of conjugates $\gamma_i^\ast$ and $\gamma_i^{\ast \ast}$ (for $i = 1, 2$) of the two representatives of $\gamma$ under $\tau$ stabilise orthogonal decompositions of $V$ into two odd-dimensional subspaces, and hence do not contain $y$. 
    Thus for all $m$ and $q$ we conclude that \[|C_{\Aut(T)}(y)| = |C_{\mathrm{PGO}_{2m}^+(q)}(y)| = q^m-1.\]
    
    Next, we observe that any involution $t_0$ inverting $y$ interchanges $E$ and $F$, and the obvious such $t_0$ is a product of $m$ reflections, and so lies in $\mathrm{PO}_{2m}^+(q)$. Hence $I_{\Aut(T)}(y)  = \langle C_{\Aut(T)}(y), t_0 \rangle \leqs \mathrm{PGO}_{2m}^+(q)$.  
    
	To complete the proof, we use Lemma~\ref{l:k=2_probability} to show that if $T \not\in \mathcal{A}$ then
    $\widetilde{Q}(T,y) < 1$, and hence $\mathcal{G}(G) = 3$. 
	Let $z\in I_{\Aut(T)}(y) \leqs \mathrm{PGO}_{2m}^+(q)$ be of prime order $r$. We will show that
	\begin{equation}\label{e:o+_b}
	|z^{\Aut(T)}| \geqs \frac{q^{2m-2}(q^m-1)(q^{m-1}-1)}{2(q+1)} =: b.
	\end{equation}
	
    To prove \eqref{e:o+_b}, first assume $r = 2$. It is shown in \cite[Proposition 2.7(vi)]{BH_Sylow} that if $q$ is odd then every involution in $\Aut(T)$, other than those of type $\gamma_1$, satisfies \eqref{e:o+_b}: involutions of type $\gamma_1$ centralise an $(n-1)$-space. If $q$ is odd then each choice of $t_0$ has a $(-1)$-eigenspace of dimension $m$, and hence $I_{\Aut(T)}(y)$ contains no involution of type $\gamma_1$, and \eqref{e:o+_b} follows.
    If $q$ is even then every involution $z$ in $I_{\Aut(T)}(y)$ lies in $I_{\Aut(T)}(y) \setminus C_{\Aut(T)}(y)$, and so is a possible $t_0$. By \cite[Proposition 3.5.16]{BG_classical} with $s = m = n/2$, 
    \[|C_{\Omega^+_{2m}(q)}(z)| \in \{q^{m(m-1)/2}|\Sp_m(q)|, q^{m(m-1)/2}|\Sp_{m-1}(q)|, q^{m(m+1)/2 - 1}|\Sp_{m-2}(q)|\}.\] 
    Hence $|C_{\mathrm{PO}^+_{2m}(q)}(t_0)| \leqs q^{m(m-1)/2 + m/2(m/2+1)} = q^{3m^2/4}$, and we conclude that \[|z^{\Aut(T)}| \geqs |z^T| \geqs q^{(5/4)m^2 - m} \geqs b.\]

     Assume next that $r$ is odd, and let $i = \Phi(r,q)$. In this case $z\in C_{\mathrm{PSO}_{2m}^+(q)}(y) = \la \widehat{y}\ra$, so $z$ stabilises the totally singular decomposition $E\oplus F$. It follows that $z$ is of type $[\Lambda^{2m/i}]$ if $i$ is even, and of type $[(\Lambda,\Lambda^{-1})^{m/i}]$ if $i$ is odd. By \cite[Table B.12]{BG_classical} 
	\begin{equation*}
	|C_{\mathrm{Inndiag}(T)}(z)| \in \{ |\GU_{2m/i}(q^{i/2})|, |\GL_{m/i}(q^i)|\},
	\end{equation*}
	so $|C_{\mathrm{Inndiag}(T)}(z)| \leqs |\GU_m(q)| \leqs (q+1)q^{m^2-1}$ by Lemma~\ref{l:order}(i), and so
	\begin{equation*}
	|z^{\Aut(T)}| \geqs |\mathrm{Inndiag}(T):C_{\mathrm{Inndiag}(T)}(z)| \geqs \frac{\frac{1}{2}q^{m(2m-1)}}{(q+1)q^{m^2-1}} > b.
	\end{equation*}
	This shows that \eqref{e:o+_b} holds in all cases.

    Next,  let $a = 2(q^m-1)$, so that 
	$|I_{\Aut(T)}(y)| \leqs a$.
	Finally, if $(m,q)\in\{(4,4),(5,2),(6,2)\}$ then let $\omega = |\Out(T)|$, otherwise let $\omega = 24\log q \geqs |\Out(T)|$. With the notation of Lemma~\ref{l:k=2_probability} we now see that 
	\begin{equation*}
	\begin{aligned}
	\widetilde{Q}(T,y) &< |I_{\Aut(T)}(y)| \cdot |\Out(T)| \cdot ab^{-1} 
	\leqs \omega \cdot a^2b^{-1} \\
	& = \frac{\omega \cdot 4(q^{m}-1)^2\cdot 2(q+1)}{q^{2m-2}(q^m-1)(q^{m-1}-1)} < 1
	\end{aligned}
	\end{equation*}
	in all cases with $(m, q) \not \in \{(4, 2), (4, 3)\}$, i.e. $T\notin\mathcal{A}$. Hence the result follows from Lemmas~\ref{l:comp} and \ref{l:k=2_probability}.
\end{proof}

{\small
	\centering
    \begin{table}
		\caption{Pairs $(T,y)$ with $T\notin \mathcal{A}$ an exceptional group}
		\begin{tabular}{@{}lll@{}}
			\toprule
			$T$ & $|y|$  & Condition \\ \midrule
			${^2}\mathrm{B}_2(q)$ & $q+\sqrt{2q}+1$ &\\
			${^2}\mathrm{G}_2(q)$ & $q+\sqrt{3q}+1$ &\\
			${^2}\mathrm{F}_4(q)'$ & $q^2+\sqrt{2q^3}+q+\sqrt{2q}+1$ & \\
			$\mathrm{G}_2(q)$ & $q^2-q+1$ & \\
			${^3}\mathrm{D}_4(q)$ & $q^4-q^2+1$ &\\
			$\mathrm{F}_4(q)$ & $q^4-q^2+1$ &$q\geqslant 3$\\
			& $17$ &$q=2$\\
			$\mathrm{E}_6^\e(q)$ & $(q^6+\e q^3+1)/(3,q-\e)$ &\\
			$\mathrm{E}_7(q)$ & $(q+1)(q^6-q^3+1)/(2,q-1)$ &$q\geqslant 3$\\
			& $129$ &$q=2$ \\
			$\mathrm{E}_8(q)$ & $q^8+q^7-q^5-q^4-q^3+q+1$ & \\
			\bottomrule
		\end{tabular}
		\label{tab:excep_uni}
	\end{table}
}

\begin{lem}
    \label{l:k=2_greedy_excep}
    Suppose $T$ is an exceptional group of Lie type that is not isomorphic to $\LL_2(q)$ for any $q$. Then $\mathcal{G}(G) = 3$.
\end{lem}

\begin{proof}
    Let $y\in T$ be a generator of a cyclic maximal torus of $T$ of order given in Table~\ref{tab:excep_uni}, which comes from \cite[Table~2]{LL_chiral}.
    We seek to apply Lemma~\ref{l:k=2_probability}. Note that
    \begin{equation*}
        \widetilde{Q}(T,y) \leqslant \frac{|\mathrm{Out}(T)|\cdot |I_{\mathrm{Aut}(T)}(y)|^2}{\min\{|z^{\mathrm{Aut}(T)}|: z\in \mathcal{P}\}},
    \end{equation*}
    so we only need to show that if $z \in \Aut(T)$ has prime order then
    \begin{equation}
    \label{e:prime_excep}
        |z^{\Aut(T)}| > |\mathrm{Out}(T)|\cdot |I_{\mathrm{Aut}(T)}(y)|^2.
    \end{equation}

    Let $d = |\mathrm{Inndiag}(T):T|$. Then, as discussed in \cite[Section 3]{LL_chiral}, the group $C_{\Aut(T)}(y)$ 
    has order $d|y|$, so $|I_{\Aut(T)}(y)|\leqs 2d|y|$.  By Lemma~\ref{l:comp} we may assume that $T \not\in \mathcal{A}$, so in particular $T \neq \mathrm{G}_2(2)' \cong \UU_3(3)$, 
    and $T \neq  {^2}\mathrm{G}_2(3)' \cong \LL_2(8)$ by our assumption that $T$ is not isomorphic to a two-dimensional linear group. Hence the bounds for $\min\{|z^{\mathrm{Aut}(T)}|:z\in \mathcal{P}\}$ given in \cite[Proposition 2.11]{BT_extremely} show that \eqref{e:prime_excep} holds for $y$.    
    
    For example, if $T = \mathrm{E}_7(q)$ with $q\geqs 3$, then $d = (2,q-1)$ and \cite[Proposition 2.11]{BT_extremely} shows that if $z \in \mathcal{P}$ then $|z^{\Aut(T)}| > q^{34}$. This gives 
    \begin{equation*}
        |z^{\Aut(T)}| > q^{34} > 4(q+1)^2(q^6-q^3+1)\log q \geqs |\Out(T)|\cdot |I_{\Aut(T)}(y)|^2
    \end{equation*}
    for any $q\geqs 3$ and any $z\in \Aut(T)$ of prime order, and so \eqref{e:prime_excep} holds. 
\end{proof}

\begin{proof}[Proof of Theorem~\ref{thm:greedy}(i)]
    The result follows from Lemma~\ref{l:k=2_greedy_alt_spo} if $T$ is alternating or sporadic, from Lemma~\ref{l:k=2_L2} if $T$ is isomorphic to $\PSL_2(q)$ for some $q$, from Lemmas~\ref{l:luso} and \ref{l:o+} if $T$ is classical in dimension at least three, and  Lemma~\ref{l:k=2_greedy_excep} if $T$ is exceptional. In particular,  $b(G) = \mathcal{G}(G)$ for each $T$ (see Theorem \ref{t:H_diag}(i)), and hence every greedy base for $G$ has the same size.
\end{proof}

\section{Greedy bases for the groups with $k \geqs 3$}\label{sec:k_ge_3}

In this section we consider the diagonal type groups $G$ with $k \geqs 3$, 
and hence complete the proof of Theorem~\ref{thm:greedy}.

\subsection{Some important partitions}\label{subsec:partitions}

We now introduce several important partitions of $[k]$ and analyse their stabilisers in $\Sn_k$ and $\An_k$. Recall that in Theorem~\ref{thm:greedy} we define $\ell = \lceil \log_{|T|} k \rceil$. We define $m:= \lceil k/|T| \rceil$, so that if $\ell \ge 2$ then $k \in \{(m-1)|T| + 1, \ldots, m|T|\}$, noting that $m \in \{|T|^{\ell-2} + 1, \ldots, |T|^{\ell-1}\}$. 

We say a partition $\Pi$ is of \textit{type} $[a_1^{b_1},\dots,a_s^{b_s}]$, written as 
\begin{equation*}
\Pi\sim [a_1^{b_1},\dots,a_s^{b_s}],
\end{equation*}
if  $0 \le a_1 < a_2 < \cdots < a_s$, for each $i$ there are precisely $b_i$ parts in $\Pi$ of size $a_i$, 
and these are all of the parts of $\Pi$. We also write $\Pi_1\sim\Pi_2$ if $\Pi_1$ and $\Pi_2$ are of same type.

Let 
$\minPart$
be a partition of $[k]$ into $|T|$ parts so that all parts have size $m$ or $m-1$,
noting that  $\minPart$ is well-defined up to $\sim$. 
Secondly,  for $k \geqs |T|+1$ we let $\Sigma$ be a partition of $[k]$ such that
\begin{equation}\label{eq:sigma}
\Sigma\sim
\begin{cases}
[(m-2)^1,(m-1)^{|T|-3},m^2] & \mbox{if $k = (m-1)|T|+1$}\\
[(m-2)^1,(m-1)^{|T|-4},m^3] & \mbox{if $k = (m-1)|T|+2$}\\
\minPart & \mbox{if $(m-1)|T|+3\leqs k\leqs m|T|-3$}\\
[(m-1)^3,m^{|T|-4},(m+1)^1] & \mbox{if $k = m|T|-2$}\\
[(m-1)^2,m^{|T|-3},(m+1)^1] & \mbox{if $k = m|T|-1$}\\
[(m-1)^2,m^{|T|-4},(m+1)^2] & \mbox{if $k = m|T|$.}\\
\end{cases}
\end{equation}

For a subgroup $H$ of $\Sn_k$ and partition $\Pi = \{\pi_t : t \in T\}$ of $[k]$, where some of the parts $\pi_t$ are allowed to be empty, we write $H_{\{\pi_t\}}$ for the setwise stabiliser of the subset $\pi_t$ of $[k]$, 
write
\begin{equation*}
H_{(\Pi)} = \bigcap_{t \in T} H_{\{\pi_t\}}
\end{equation*}
for the setwise stabiliser of each part of $\Pi$, and write $H_{\{\Pi\}}$ for the subgroup of $H$ that permutes the parts of $\Pi$.

We show next that if $k \ge |T| + 1$ then $\minPart$ is the unique type of partition with smallest stabiliser in $H \in \{\An_k, \Sn_k\}$.

\begin{lem}\label{l:min_part}
Suppose $k \ge |T| + 1$, and let $\Pi$ be a partition of $[k]$ into exactly $|T|$ 
parts (some of which may be empty) such that $\Pi \not\sim \minPart$. Let the smallest and largest parts of $\Pi$ have size $d$ and $e$, respectively. Let $H$ be $\An_k$ or $\Sn_k$ acting naturally on $[k]$. Then \[|H_{(\Pi)}| \geqs \frac{e}{d+1}|H_{(\minPart)}| >|H_{(\minPart)}|,\] with equality in the first bound 
if and only if $\Pi$ has either a unique part not of size $m$ or $m-1$, and this part has size $m-2$ or $m+1$, or exactly two such parts, one of size $m-2$ and one of size $m+1$. 
\end{lem}

\begin{proof}
   Suppose for now that $H = \Sn_k$.
   Let $\pi_1$ and $\pi_2$ be parts of $\Pi$ of size $d$ and $e$, respectively. Consider the partition $\Pi_1$ which is identical to $\Pi$ except one point from $[k]$ has been moved from $\pi_2$ to $\pi_1$. 
   Then \[\frac{|H_{(\Pi)}|}{|H_{(\Pi_1)}|} = \frac{d!e!}{(d+1)!(e-1)!} = \frac{e}{d+1}.\]
   Since not all parts of $\Pi$ have size $m$ or $m-1$, the difference $e-d \geqs 2$, so this is strictly greater than $1$. 
   If $\Pi_1 \sim \minPart$ then we are done, otherwise we may iterate this process to produce a sequence of partitions with decreasing stabiliser orders until no part has size greater than $m$ or less than $m-1$, namely a partition equivalent to $\minPart$.

   If $H = \An_k$, then from $k \geqs |T| + 1$ we see that at least one part of each $|T|$-part partition has size at least two, and so the stabiliser in $\Sn_k$ contains a transposition. Hence the stabilisers in $\An_k$ have order precisely half that of the corresponding stabilisers in $\Sn_k$, and the ratio of stabiliser orders is unchanged. 
\end{proof}

The next lemma shows that if $k \ge |T| + 1$ then,  up to $\sim$, the partition $\Sigma$ has at worst the third smallest stabiliser in $H \in \{\An_k, \Sn_k\}$ amongst all partitions of $[k]$, and that its stabiliser is less than twice as big as the smallest.

\begin{lem}
	\label{l:part}
	Suppose that $k \ge |T| + 1$ and  let $\Pi$ be a partition of $[k]$ into $|T|$ 
    parts such that $\Pi \not\sim \minPart$,  $\Pi \not\sim \Sigma$ and $\Pi \not \sim [(m-1)^1, m^{|T|-2}, (m+1)^{1}]$.  Let $H$ be $\An_k$ or $\Sn_k$ as before. Then $|H_{(\Pi)}| > |H_{(\Sigma)}|$. Furthermore, $|H_{(\Sigma)}| \leqs 2|H_{(\minPart)}|$.
\end{lem}

\begin{proof}
	If $k \in \{(m-1)|T| + 3, \ldots, m|T|-3\}$ then $\Sigma\sim\Gamma_{k, |T|}$ and this is immediate from Lemma~\ref{l:min_part}, so suppose otherwise. Let $d$ and $e$ be the sizes of the smallest and largest parts of $\Pi$, respectively. Then Lemma~\ref{l:min_part} shows that $$|H_{(\Pi)}| \geqs \frac{e}{d+1}|H_{(\minPart)}|,$$ with equality only under very restricted conditions.  We compare this bound with the value of $|H_{(\Sigma)}|$ for the remaining values of $k$.

    Suppose $k = (m-1)|T| + a$ with $a = 1, 2$, where $\minPart \sim [(m-1)^{|T|-1}, m]$ or $\minPart \sim [(m-1)^{|T|-2}, m^2]$, respectively.
    Then Lemma~\ref{l:min_part} shows that
    $$|H_{(\Sigma)}| = \frac{m}{m-1}|H_{(\minPart)}| \leqs 2|H_{(\minPart)}|.$$
    Now $e \geqs m$, so if $d < m-2$ then the result follows for all $e$. Furthermore, $d \leqs m-2$ (since $k = (m-1)|T|+a$), so if $e \geqs m+1$ then the result follows. The only remaining case is $d = m-2$ and $e = m$, but now $\Pi \not\sim\Sigma$ shows that $\Pi$ has at least $a + 2$ parts of size $m-2$, so $|H_{(\Pi)}| > |H_{(\Sigma)}|$ in this case as well. 

    Our next cases are  $k = m|T|-a$ with $a = 2, 1, 0$, where $\minPart \sim [(m-1)^2, m^{|T|-2}]$, $ \minPart \sim [(m-1), m^{|T|-1}]$, or $\minPart \sim [m^{|T|}]$, respectively. This time Lemma~\ref{l:min_part} shows that  $$|H_{(\Sigma)}| = \frac{m+1}{m}|H_{(\minPart)}|  <  2|H_{(\minPart)}|. $$ 
    Now $e \geqs m$, so if $d \leqs  m-2$ then the result follows. Hence we may assume that $d = m-1$, and so (since $\Pi \not\sim \minPart$) we deduce that $e \geqs m+1$. Furthermore, if $e \geqs m+2$ then the result follows, so we may assume that $d =  m-1$ and $e = m+1$. But $\Pi \not\sim \Sigma$ and $\Pi \not\sim [(m-1)^1, m^{|T|-2}, (m+1)^{1}]$, so $\Pi$ has at least $2$ parts of size $m+1$ if $a \in \{1, 2\}$, and at least three if $a = 0$. If $a \in \{1, 2\}$ this completes the argument, so assume that $a = 0$ and $\Pi \sim [(m-1)^\alpha, m^{|T|-2\alpha},(m+1)^{\alpha}]$ for some $\alpha > 2$. Then 
    \[\frac{|H_{(\Pi)}|}{|H_{(\Sigma)}|} = \frac{((m-1)!)^{\alpha} (m!)^{|T|-2\alpha}((m+1)!)^{\alpha}}
    {((m-1)!)^{2} (m!)^{|T|-4}((m+1)!)^{2}} = \frac{((m-1)!)^{\alpha-2} ((m+1)!)^{\alpha-2}}{(m!)^{2\alpha-4}} 
    = \frac{(m+1)^{\alpha - 2}
    }{m^{\alpha-2}}\] which is strictly greater than $1$, as required.
    \end{proof}

\subsection{Proof of Theorem~\ref{thm:greedy}(iii)}

Recall from the introduction the groups $P \le \Sn_k$ and $Q \le G \cap \Sn_k$, and recall \eqref{e:GD} that
\begin{equation*}
    G_D = D \leqs \{(\varphi,\dots,\varphi)\sigma:\varphi\in\Aut(T),\ \sigma\in P\}.
\end{equation*}
We now analyse the two-point stabilisers in $G$. Throughout this section we shall assume that $\An_k \le P$. 
For each $\mathbf{t} := (t_1,\dots,t_k)\in T^k$, there is an associated partition $\mathcal{P}^{\mathbf{t}}$ of $[k]$ with parts $\{\mathcal{P}_s^{\mathbf{t}} : s\in T\}$
such that $j\in\mathcal{P}_s^{\mathbf{t}}$ if $t_j = s$: some parts of $\mathcal{P}^{\mathbf{t}}$ may be empty. We say that $s \in T$ is the \emph{label} of $\mathcal{P}^{\mathbf{t}}_s \subset [k]$, and write $P_{\mathcal{P}^{\mathbf{t}}_s}$ for the setwise stabiliser in $P$ of $\mathcal{P}^{\mathbf{t}}_s$.

\begin{lem}[{\cite[Lemma 3.3]{FHLR_Saxl},  \cite[Lemma 2.2]{H_diag}}]\label{l:H_diag_l:2.2}
	Let $k\geqs 2$, $\mathbf{t} = (t_1,\dots,t_k)\in T^k$, and let $\mathcal{P}^{\mathbf{t}}$ be as above. Suppose that $(\varphi,\dots,\varphi)\sigma\in G_{D\mathbf{t}}$. Then the following properties hold.
	\begin{itemize}\addtolength{\itemsep}{0.2\baselineskip}
		\item[{\rm (i)}] $\sigma\in P$ permutes the parts of $\mathcal{P}^{\mathbf{t}}$. 
        \item[{\rm (ii)}] If $\mathcal{P}_1^{\mathbf{t}} \neq \emptyset$, and $|\mathcal{P}_1^{\mathbf{t}}|\ne |\mathcal{P}_s^{\mathbf{t}}|$ for all $s \in T \setminus \{1\}$, then 
        \begin{enumerate}
            \item[(a)] $t_j^\varphi = t_{j^\sigma}$ for all $j$; and 
		    \item[(b)] for any $p\in \mathbb{Z}_{\geqs 0}$, the element $\varphi \in \Aut(T)$ setwise stabilises the set of elements of $T$ labelling parts of $\mathcal{P}^{\mathbf{t}}$ of size $p$.
	    \end{enumerate}
	\end{itemize}
\end{lem}


\begin{cor}
	\label{c:stab_2-parts}
	Let $\mathbf{t} \in T^k$ and $\mathcal{P}^{\mathbf{t}}$ be as above. 
    \begin{itemize}\addtolength{\itemsep}{0.2\baselineskip}
        \item[{\rm (i)}] The group $Q_{(\mathcal{P}^{\mathbf{t}})}$ is a subgroup of the two-point stabiliser $G_{D , D\mathbf{t}}$.
        \item[{\rm (ii)}] If $G = W$ and $\mathcal{P}^{\mathbf{t}}$ has precisely two non-empty parts, labelled $1$ and $t$ and of distinct sizes, then 
	\begin{equation*}
	\begin{aligned}
	G_{D, D\mathbf{t}} &= \{(\varphi,\dots,\varphi)\sigma: \varphi\in C_{\Aut(T)}(t), \sigma\in P_{\mathcal{P}_1^{\mathbf{t}}, \mathcal{P}_t^{\mathbf{t}}}\} \cong C_{\Aut(T)}(t)\times P_{\mathcal{P}_1^{\mathbf{t}}}.
	\end{aligned}
	\end{equation*}
    \end{itemize}
\end{cor}

\begin{proof}
    For (i), if $(1,\dots,1)\sigma = \sigma \in Q_{(\mathcal{P}^{\mathbf{t}})}$ then $t_{i^\sigma} = t_i$ for all $i \in [k]$. Then \eqref{e:GD} shows that $\sigma\in G_{D\mathbf{t}}$, and so $Q_{(\mathcal{P}^{\mathbf{t}})} \leqs G_{D,D\mathbf{t}}$. Part (ii) is an immediate corollary of Lemma \ref{l:H_diag_l:2.2}.
\end{proof}

Let $\ell = \lceil\frac{\log k}{\log |T|}\rceil$. We first focus on the case where $k\geqs |T|+1$ (so $\ell\geqs 2$). Recall our assumption that $\An_k \le P$. In this setting, by \cite[Corollary 2.6]{H_diag}, 
\begin{equation}\label{eq:wreath}
    \mbox{the group $T \wr \An_k$ is a subgroup of $G$, and so $\An_k \le Q \le \Sn_k$.}
\end{equation}

To begin, we show that there is an element of $\Omega$, corresponding to a partition of type $\Sigma$ (as in \eqref{eq:sigma}), such that the resulting two-point stabiliser in $G$ is precisely the group $Q_{(\Sigma)} \leqs \Sn_k$: we shall then show  that 
the greedy algorithm chooses a second base point of this sort.

\begin{lem}
	\label{l:two_stab_sigma}
	Suppose $k\geqs |T|+1$. Then there exists an $\mathbf{s}\in T^k$ with $\mathcal{P}^{\mathbf{s}} = \Sigma$ such that $G_{D,D\mathbf{s}} = Q_{(\Sigma)}$.
\end{lem}

\begin{proof}
	In view of \eqref{eq:sigma}, let $m' = m-1$ if $k \in \{(m-1)|T|+1, (m-1)|T|+2\}$, otherwise let $m' = m$. Let $b$ be the number of parts of $\Sigma$ of size $m'$.  Then by \cite[Theorem 4]{H_diag}, since $3\leqs b\leqs |T|-3$, there exists an $\mathbf{s} = (t_1,\dots,t_k)\in T^k$ with $\mathcal{P}^{\mathbf{s}} = \Sigma$ such that the setwise stabiliser of $\{t \in T :|\mathcal{P}_t^{\mathbf{s}}| = m'\}$ in $\Hol(T)$ is trivial. 
	
	Let $(\varphi,\dots,\varphi)\sigma\in G_{D,D\mathbf{s}}$, so that $\sigma\in P_{\{\Sigma\}}$ by Lemma~\ref{l:H_diag_l:2.2}(i). In particular, this implies that $\sigma$ setwise stabilises 
    $\{\mathcal{P}_t^{\mathbf{s}} : |\mathcal{P}_t^{\mathbf{s}}| = m'\}$. 
    There exists a unique $g\in T$ such that $t_j^\varphi = gt_{j^\sigma}$ for all $j\in [k]$. Hence, if $|\mathcal{P}_t^{\mathbf{s}}| = m'$, then $|\mathcal{P}_{g^{-1}t^\varphi}^{\mathbf{s}}| = m'$. Thus, $g^{\varphi^{-1}}\varphi$ is in the setwise stabiliser in $\Hol(T)$ of $\{t \in T :|\mathcal{P}_t^{\mathbf{s}}| = m'\}$, which is trivial by the choice of $\mathbf{s}$. It follows that $\varphi = 1$ and $g = 1$, and so $t_j = t_{j^\sigma}$ for all $j\in [k]$. This shows that $\sigma\in P_{(\Sigma)}\cap G = Q_{(\Sigma)}$, which yields $G_{D,D\mathbf{s}}\leqs Q_{(\Sigma)}$. Corollary~\ref{c:stab_2-parts}(i) showed that $Q_{(\Sigma)}\leqs G_{D,D\mathbf{s}}$, so the proof is complete. 
\end{proof}

 Now as promised we show that the greedy algorithm chooses a second base point corresponding to a partition of type $\Sigma$. This will be a key tool in what follows. 
 
\begin{lem}
	\label{l:greedy_two_stab}
	Suppose $k\geqs |T|+1$, and let $\mathbf{t}\in T^k$ be such that $|G_{D,D\mathbf{t}}|$ is minimum. Then $\mathcal{P}^{\mathbf{t}}\sim\Sigma$ and $G_{D,D\mathbf{g}} = Q_{(\mathcal{P}^{\mathbf{t}})}$. 
\end{lem}

\begin{proof}
	By Lemma~\ref{l:two_stab_sigma},  there exists an $\mathbf{s}\in T^k$ with $\mathcal{P}^{\mathbf{s}} = \Sigma$ and $|G_{D,D\mathbf{s}}|  =  |Q_{(\Sigma)}|$. Let $\mathbf{t} = (g_1,\dots,g_k) \in T^k$ be such that $\mathcal{P}^{\mathbf{t}} \not \sim\Sigma$. It suffices to  show that $|G_{D,D\mathbf{t}}| > |Q_{(\Sigma)}|$.

    Corollary~\ref{c:stab_2-parts}(i) shows $G_{D,D\mathbf{t}} \geqs Q_{(\mathcal{P}^{\mathbf{t}})}$, and Lemma~\ref{l:part}  shows that if $\mathcal{P}^{\mathbf{t}}\not \sim \minPart$ and $\mathcal{P}^{\mathbf{t}} \not \sim [(m-1)^1, m^{|T|-2}, (m+1)^{1}]$ then 
	$|Q_{(\mathcal{P}^{\mathbf{t}})}| > |Q_{(\Sigma)}|$, which is equal to $|G_{D,D\mathbf{s}}|$.
	Thus, we may assume that either $\mathcal{P}^{\mathbf{t}}\sim \minPart$ (with five possibilities for $k$ modulo $|T|$), or $k = m|T|$ and $\mathcal{P}^{\mathbf{t}} \sim [(m-1)^1, m^{|T|-2}, (m+1)^{1}]$.
	
    First suppose  $k = m|T| +  \epsilon 1$, with $\epsilon = \pm$. Then $\mathcal{P}^{\mathbf{t}}\sim \minPart \sim [(m + \epsilon 1)^1,m^{|T|-1}]$. Without loss of generality, we may assume $|\mathcal{P}_1^{\mathbf{t}}| = m + \epsilon 1$. Then for any $x\in T$,
	\begin{equation*}
	\{t:|\mathcal{P}_t^{\mathbf{t}}| = m\} = T\setminus\{1\} = (T\setminus\{1\})^{x} = \{t:|\mathcal{P}_t^{\mathbf{t}}| = m\}^{x}.
	\end{equation*}
	Since $m \ge 2$ there exists $\sigma\in \An_k$ such that
	\begin{equation*}
	j^\sigma = i\mbox{ if and only if }g_j^{x} = g_i.
	\end{equation*}
	Hence \eqref{eq:wreath} shows that $(x,\dots,x)\sigma\in T \wr \An_k \leqs G$, and  we readily check that
	\begin{equation*}
	D^{(x, \ldots, x)\sigma} = D \quad \mbox{ and } \quad  D\mathbf{t}^{(x,\dots,x)\sigma} = D(g_{1^{\sigma^{-1}}}^{x},\dots,g_{k^{\sigma^{-1}}}^{x}) = D(g_1,\dots,g_k) = D\mathbf{t}.
	\end{equation*}
	Hence $|G_{D,D\mathbf{t}}| \geqs |T|\cdot |Q_{(\mathcal{P}^{\mathbf{t}})}| > 2|Q_{(\mathcal{P}^{\mathbf{t}})}| \geqs |Q_{(\Sigma)}|$ by Lemma~\ref{l:part}, as required. 
	
	Next, suppose $k = m|T| + \epsilon 2$, so $\mathcal{P}^{\mathbf{t}}\sim [(m + \epsilon 1 )^2,m^{|T|-2}]$, and assume that $|\mathcal{P}_1^{\mathbf{t}}| = m + \epsilon 1$. Let $z\in T\setminus\{1\}$ be the other element with $|\mathcal{P}_z^{\mathbf{t}}| = m + \epsilon 1$. Again, for any $x\in C_T(z)$,
	\begin{equation*}
	\{t:|\mathcal{P}_t^{\mathbf{t}}| = m\} = T\setminus\{1,z\} = (T\setminus\{1,z\})^{x} = \{t:|\mathcal{P}_t^{\mathbf{t}}| = m\}^{x}.
	\end{equation*}
	Hence again there exists $\sigma\in \An_k \le Q$ such that
	\begin{equation*}
	j^\sigma = i\mbox{ if and only if }g_j^{x} = g_i,
	\end{equation*}
	and it is easy to check that $(x,\dots,x)\sigma\in G_{D,D\mathbf{t}}$. Since no element of a non-abelian simple group has a  centraliser of order $2$,  
    it follows that $|G_{D,D\mathbf{t}}|\geqs |C_T(z)|\cdot |Q_{(\mathcal{P}^{\mathbf{t}})}| > 2|Q_{(\mathcal{P}^{\mathbf{t}})}| \geqs |Q_{(\Sigma)}|$ by Lemma~\ref{l:part}, as required.
	
	Finally, if $k = m|T|$, then either $\mathcal{P}^{\mathbf{t}}\sim \minPart \sim [m^{|T|}]$ or $\mathcal{P}^{\mathbf{t}}\sim [(m-1)^1,m^{|T|-2},(m+1)^1]$. We leave it as an exercise in each case to show that $|G_{D,D\mathbf{t}}| > 2|Q_{(\mathcal{P}^{\mathbf{t}})}| \geqs |Q_{(\Sigma)}|$.
\end{proof}

\begin{lem}\label{l:ceil_divide}
    Let $m, n, r \in \mathbb{Z}_{\geqs 0}$. Then $\lceil m/n^r \rceil/n = \lceil m/n^{r+1} \rceil$.
\end{lem}

\begin{proof}
    Write $m = an^{r+1} + bn^{r} + c$, where $b \in \{0, \ldots, n-1\}$ and $c \in \{0, \ldots, n^r-1\}$.
If $c = 0$ then $\lceil m/n^r \rceil/n = (an + b)/n = a + \lceil b/n \rceil = \lceil m/n^{r+1}\rceil$. Otherwise $\lceil m/n^r \rceil/n = (an + b + 1)/n = a + \lceil (b+1)/n \rceil = a + 1 = \lceil m/n^{r+1} \rceil$. 
\end{proof}

\begin{lem}\label{lem:inductive_step}
    Suppose $k\geqs |T|+1$. Let $\beta_i = (D, D\mathbf{t}_1, \ldots, D\mathbf{t}_i)$ be the first $i+1$ points chosen by the greedy algorithm, for some $i \in \{1, \ldots, \ell+1\}$. Let
    \begin{equation*}
        m' = 
        \begin{cases}
            m & \mbox{if $k\leqs m|T|-3$};\\
            m+1 & \mbox{if $k\in\{m|T|-2, m|T|-1, m|T|\}$}.
        \end{cases}
    \end{equation*}
    Then $G_{\beta_i} = Q_{(\Pi_i)}$, where $\Pi_i$ is a partition with largest part of size 
    $\lceil m'/|T|^{i-1} \rceil$.
    Furthermore, if $(k \pmod{|T|^i}) \not\in \{-1,-2\}$ then there are at least two parts of this largest size.
\end{lem}

\begin{proof}
We prove the result by induction on $i$. 
For $i = 1$,  Lemma~\ref{l:greedy_two_stab} shows that $G_{\beta} = Q_{(\Pi_1)}$
for some partition $\Pi_1 \sim \Sigma$. We see from \eqref{eq:sigma} that the largest part of $\Pi_1$ has size $m'$, 
and that if $k \not\in \{m|T|-1, m|T|-2\}$ then $\Pi_1$ has at least two parts of size $m'$.

Now suppose $i > 1$ and $G_{\beta_{i-i}} = Q_{(\Pi_{i-1})}$, where  
the  
largest part of $\Pi_{i-1}$ has size $\lceil m'/|T|^{i-2} \rceil$. 
Then $G_{\beta_i} = G_{\beta_{i-1},D,D\mathbf{t}_i} = Q_{(\Pi_{i-1})} \cap Q_{(\Pi')}$, where $\Pi'$ is a partition of $[k]$ into $|T|$ parts such that the common refinement of $\Pi_{i-1}$ and $\Pi'$ has minimal order stabiliser in $Q$. 
Since $Q_{(\Pi_{i-1})} =  (\Sn_{a_1} \times \cdots \times \Sn_{a_\ell}) \cap Q$ is 
either $\Sn_{a_1} \times \cdots \times \Sn_{a_\ell}$ or the intersection of this with $\A_k$, it follows from Lemma~\ref{l:min_part} that $\Pi_i$ splits the $j$th part of $\Pi_{i-1}$ into $|T|$ parts, of sizes $\lfloor a_j/|T| \rfloor$ and $\lceil a_j/|T| \rceil$. In particular the largest part has size 
\[\left \lceil \frac{\lceil m'/|T|^{i-2} \rceil}{|T|} \right \rceil,\] which is equal to $\lceil m'/ |T|^{i-1} \rceil$ by Lemma~\ref{l:ceil_divide}. 

For the final claim, if $(k \mod |T|^{i-1}) \not\in \{-1, -2\}$ then there were at least two parts of the largest size in $\Pi_{i-1}$, and hence there will be at least two parts of the largest size in $\Pi_{i}$. Otherwise, if $\lceil m'/|T|^{i-1} \rceil  = \lceil (m+1)/|T|^{i-1} \rceil \neq a|T| + 1$ for some $a \in \mathbb{Z}$ there will also be at least two parts of the largest size in $\Pi_{i}$. Hence there is a unique part of the largest size if and only if $m/|T|^{i-1}$ is divisible by $|T|$, so that $m$ is divisible by $|T|^{i}$, and $(k \pmod{|T|^i}) \in \{-1, -2\}$, as required.
\end{proof}

\begin{prop}\label{p:|T|+1}
    Suppose $k\geqs |T|+1$. Then every greedy base has size $\mathcal{G}(G) \in \{\ell + 1, \ell + 2\}$, with $\mathcal{G}(G) = \ell + 2$ if and only if either $k = |T|^{\ell}$, or $Q = \Sn_k$ and $k \in \{ |T|^{\ell}-1, |T|^{\ell}-2\}$. 
\end{prop}

\begin{proof}
    By Lemma~\ref{lem:inductive_step}, after the greedy algorithm has chosen $\ell$ base points,
    the remaining 
    stabiliser is $Q_{(\Pi_{\ell-1})}$, for some partition $\Pi_{\ell-1}$ of $[k]$ into parts of size at most $\lceil m'/|T|^{\ell-2} \rceil \geqs 2$. This is equal to $2$ only when $m \leqs 2|T|^{\ell-2} < |T|^{\ell}-2$, so there are at least two parts of the largest size. After one more base point is fixed, 
    the largest part of $\Pi_{\ell}$ has size at most $\lceil m'/|T|^{\ell-1} \rceil$, which is $1$ unless $m' = m+1 = |T|^{\ell-1} + 1$.  This happens only when $k \geqs |T|^{\ell} - 2$ by Lemma~\ref{lem:inductive_step}, so if $k < |T|^{\ell}-2$ then every greedy base has size $\ell +1$.

    Suppose $k \geqs |T|^{\ell} -2$, and notice that $Q_{(\Pi_{\ell + 1})} = 1$ by Lemma~\ref{lem:inductive_step}. If $k = |T|^{\ell}$ then there are at least two parts of size $2$ in $\Pi_{\ell}$, so $Q_{(\Pi_{\ell})} \neq 1$, and hence every greedy base has size $\mathcal{G}(G) = \ell + 2$. If $|T|^{\ell-1} - 2 \leqs k \leqs |T|^{\ell-1} - 1$ then $\Pi_{\ell}$ has a unique part of size $2$, and so $Q_{(\Pi_{\ell})} = 1$ if and only if $Q = \An_{k}$. The result follows. 
\end{proof}

We are now ready to prove Theorem~\ref{thm:greedy}.

\begin{proof}[Proof of Theorem~\ref{thm:greedy}] We proved Part (i) at the end of Section~\ref{sec:k=2}. As noted in the introduction,
if $b(G) = 2$ then $\mathcal{G}(G) = 2$, so 
    Part (ii) is immediate from Theorem~\ref{t:H_diag}(ii), and it remains to prove Part (iii) for the groups with $b(G) > 2$. By Theorem~\ref{t:H_diag}(iii), if $b(G) > 2$ then either $k\in \{|T|-2, |T|-1\}$ and $Q = \Sn_k$, or $k\geqs |T|$. The cases with $k\geqs |T|+1$ are handled in Proposition~\ref{p:|T|+1}. For all other cases it is shown in \cite[Proposition~3.5]{FHLR_Saxl} that every pair of distinct points in $\Omega$ can be extended to a base of size $3$, which yields $\mathcal{G}(G) = 3 = b(G)$. Part (iii) follows. 
    
    For the final statement, if $\mathcal{G}(G) = b(G)$ then every greedy base has size $b(G)$. The only situation where this does not occur is discussed in  Proposition~\ref{p:|T|+1}, which proves that every greedy base has size $\mathcal{G}(G)$. 
\end{proof}

\section{Relational complexity}\label{sec:RC}

In this section we prove Theorems~\ref{thm:RC4} and \ref{thm:RClog}, and also show that the lower bound in Theorem~\ref{thm:RC4} is attained infinitely often.

\begin{proof}[Proof of Theorem~\ref{thm:RC4}]
	By \cite[Corollary 1.4]{GLS_binary}, there is no primitive group $G$ of diagonal type with $\mathrm{RC}(G) = 2$, so it suffices to show that $\mathrm{RC}(G)\ne 3$. Let $G$ have socle $T^k$.
    
     If $T\in\{ \An_5, \An_6\}$ and $k = 2$ then we may check the result using {\sc Magma} \cite{MAGMA}, 
    so we may assume that if $k = 2$ then $T\notin\{ \An_5, \An_6\}$ from now on. If $k = 2$ then by Theorem~\ref{t:H_diag}(i) the group $G$ has a base of size $3$, so let $x, y \in T$ be such that
    $(D,D(x,1),D(y,1))$ is a base. 
     If $k \geqs 3$ then let $(x,y)$ be a generating pair of $T$. For all $k$, let
	\begin{equation*}
	\begin{aligned}
	\alpha &= D\\
	\beta &= D(x,1,\dots,1)\\
	\gamma &= D(y,1,\dots,1)\\
	\delta_1 &= D(xy^{-1},1,\dots,1)\\
	\delta_2 &= D(y^{-1}x,1,\dots,1)
	\end{aligned}
	\end{equation*}
    be elements of $\Omega$, and define $I = (\alpha,\beta,\gamma,\delta_1)$, $J = (\alpha,\beta,\gamma,\delta_2)$. 
    
    We claim that $xy^{-1} \neq  y^{-1}x$: suppose otherwise. If $k = 2$ then this implies that $(x,x)$ fixes $\alpha$, $\beta$, and $\gamma$, which contradicts our assumption that $(\alpha,\beta,\gamma)$ is a base for $G$. If $k \geqs 3$ then $xy^{-1} = y^{-1}x$ contradicts $\langle x, y \rangle = G$. 
    Hence these five points are 
    distinct. We shall show that $I \sim_3 J$ but $I \not\sim_4 J$, from which the result will follow. 

    We first calculate
	\begin{equation*}
	\begin{aligned}
	\b^{(y^{-1}x,xy^{-1},\dots,xy^{-1})} &= D(xy^{-1}x,xy^{-1},\dots,xy^{-1}) = D(x,1,\dots,1) = \b\\
	\g^{(y^{-1}x,xy^{-1},\dots,xy^{-1})} &= D(x,xy^{-1},\dots,xy^{-1}) = D(y,1,\dots,1) = \g\\
	\delta_1^{(y^{-1}x,xy^{-1},\dots,xy^{-1})} &= D(xy^{-1}y^{-1}x,xy^{-1},\dots,xy^{-1}) = D(y^{-1}x,1,\dots,1) = \delta_2.
	\end{aligned}
	\end{equation*}
    We therefore deduce that $(\alpha, \beta, \gamma)^1 = (\alpha, \beta, \gamma)$, whilst
	\begin{equation*}
	\begin{aligned}
	&(\alpha,\beta,\delta_1)^{(x,\dots,x)}  = (\alpha,\beta,\delta_2)\\
	&(\alpha,\gamma,\delta_1)^{(y,\dots,y)} = (\alpha,\gamma,\delta_2)\\
	&(\beta,\gamma,\delta_1)^{(y^{-1}x,xy^{-1},\dots,xy^{-1})} = (\beta,\gamma,\delta_2)
	\end{aligned}
	\end{equation*}
    and so $I \sim_3 J$. 
	
    It remains to show that $I$ and $J$ are not in the same $G$-orbit. For $k = 2$ this follows from 
    our assumption that $(\alpha, \beta, \gamma)$ is a base. For $k \geqs 3$, we may assume that
    $G = W$. Then by Corollary \ref{c:stab_2-parts}(ii), $G_{\alpha, \beta} = C_{\mathrm{Aut}(T)}(x)\times P_1$ and $G_{\alpha,\gamma} = C_{\mathrm{Aut}(T)}(y)\times P_1$. From $\langle x,y\rangle = T$ we deduce that $G_{\alpha,\beta,\gamma} = P_1$. However, $\delta_1$ and $\delta_2$ lie in distinct $P_1$-orbits, whence $I \not\sim_4 J$.
\end{proof}

It is now easy to show that the bound in Theorem~\ref{thm:RC4} is attained infinitely often.

\begin{prop}
	\label{p:RC4_inf}
	There are infinitely many diagonal type groups $G$ with $\mathrm{RC}(G) = 4$.
\end{prop}

\begin{proof}
	Let $f\geqslant 2$ be an integer and let $G = T^2$ with $T = \mathrm{L}_2(2^f)$. Then \cite[Theorem~1.1]{LMM_ibis} shows that $I(G)= b(G)$, so $I(G)  = 3$ by Theorem~\ref{t:H_diag}. Therefore the bound $\RC(G) \leqs I(G) + 1$ yields  $\mathrm{RC}(G) \leqslant I(G) + 1 = 4$. Combining with Theorem~\ref{thm:RC4} now shows that $\mathrm{RC}(G) = 4$.
\end{proof}

We conclude this section by determining a lower bound on the relational complexity of a family of groups of diagonal type.

\begin{prop}
	\label{p:RC_to_inf}
	For all $m \ge 3$ and all $k \geqs 3$, every primitive group $G = \An_{m+2}^k{:}P$ of diagonal type satisfies $\mathrm{RC}(G) \geqs m$.
\end{prop}

\begin{proof}
	Let $T =  \An_{m+2}$. 
    For $2\leqslant i\leqslant m+1$, let $t_i = (1,2,i+1)\in T$. Define $\alpha_1 = D$, $\alpha_i = D(t_i,1,\dots,1)$ for $2\leqslant i\leqslant m-1$, $\beta = D(t_{m},1,\dots,1)$ and $\gamma = D(t_{m+1},1,\dots,1)$. Let $I = (\alpha_1, \ldots, \alpha_{m-1}, \beta) \in \Omega^m$ and $J = (\alpha_1, \ldots, \alpha_{m-1}, \gamma) \in \Omega^m$. We shall show that $I \sim_{m-1} J$ but $J \not\in I^G$, from which the result follows. 
	
	Since $k \geqs 3$, we can apply Corollary \ref{c:stab_2-parts} with parts of size $1$ and $k-1$ to see that $G_{\alpha_1, \alpha_i} = C_T(t_i)\times P_1$, and $C_T(t_2,\dots,t_{m-1}) = C_{T}(\An_{m}) = 1$.  Therefore $I$ and $J$ are in distinct $G$-orbits. For $1\leqs i\leqs m$, let $s_i = (i,m+1,m+2) \in T$. Then $s_1^{-1} t_i s_2 = t_i$ for $i \in \{2, \ldots, m-1\}$, whilst $s_1^{-1} t_m s_2 = t_{m+1}$, so
	\begin{equation*}
	(\alpha_2,\dots,\alpha_{m-1},\beta)^{(s_2,s_1,\dots,s_1)} = (\alpha_2,\dots,\alpha_{m-1},\gamma).
	\end{equation*}
	Similarly, one can check that for all $2\leqs i\leqs m-1$,
	\begin{equation*}
	(\alpha_1,\dots,\a_{i-1},\a_{i+1},\dots,\alpha_{m-1},\beta)^{(s_{i+1},s_{i+1},\dots,s_{i+1})} = (\alpha_1,\dots,\a_{i-1},\a_{i+1},\dots,\alpha_{m-1},\gamma),
	\end{equation*}
	so $I \sim_{m-1} J$, which completes the proof. 
\end{proof}

\begin{proof}[Proof of Theorem~\ref{thm:RClog}]
    By Proposition~\ref{p:RC_to_inf}, for all $m \geqs 3$ and $k \geqs 3$ there is a diagonal type group $G$ with $\RC(G) \geqs m$ and degree $n = |\An_{m+2}|^{k-1}$. Setting $k = 3$ and using the bounds $d^d/e^{d-1} \leqs d! \leqs d^{d+1}/e^{d-1}$ (see, for example, \cite[p. 52]{Knuth}), we get inequalities
    \begin{equation*}
        \left(\frac{(m+2)^{m+2}}{2e^{m+1}}\right)^2 \leqs n  \leqs \left(\frac{(m+2)^{m+3}}{2e^{m+1}}\right)^2,
    \end{equation*}
    so for sufficiently large $m$ we can bound
    \begin{align*}
      \log n & \leqs 2((m+3) \log (m+2) - (m+1) \log e - 1) \leqs 2m \log (m+2) \\
      \log \log n  & \geqs \log \left((2m + 4)\log (m+2) - (2m+2) \log e - 2) \right) \geqs \log (2m+4) \geqs \log (m+2)
    \end{align*} 
    and hence \[ \log n/ \log \log n \leqs 2m,\]
    as required.
\end{proof}


\end{document}